\documentclass[12pt]{article}
\usepackage{amssymb,amsmath,amsthm, amsfonts}
\usepackage{graphicx}
\usepackage{epsfig}
\usepackage{tikz}
\usepackage{mathrsfs}

\def\disp{\displaystyle}

\def\dref#1{(\ref{#1})}

\theoremstyle{plain}
\newtheorem{theorem}{Theorem}[section]
\newtheorem{lemma}{Lemma}[section]
\newtheorem{proposition}{Proposition}[section]

\theoremstyle{definition}

\newtheorem{remark}{Remark}[section]

\numberwithin{equation}{section}

\begin{document}

\title{\bf An optimal  result  for global classical and bounded  solutions in a two-dimensional Keller-Segel-Navier-Stokes system with    sensitivity
}

\author{Jiashan Zheng$^{a,b}$,\thanks{Corresponding author.   E-mail address:
 zhengjiashan2008@163.com (J. Zheng)}
\\
 $^{a}$
    School of Information,\\
    Renmin University of China, Beijing, 100872, P.R.China \\
 $^{b}$
    School of Mathematics and Statistics Science,\\
     Ludong University, Yantai 264025,  P.R.China \\
}
\date{}


\maketitle \vspace{0.3cm}
\noindent
\begin{abstract}
This paper deals with a boundary-value problem for a coupled
chemotaxis-Navier-Stokes system  involving tensor-valued
sensitivity with saturation
$$
 \left\{
 \begin{array}{l}
   n_t+u\cdot\nabla n=\Delta n-\nabla\cdot(nS(x,n,c)\nabla c),\quad
x\in \Omega, t>0,\\
    c_t+u\cdot\nabla c=\Delta c-c+n,\quad
x\in \Omega, t>0,\\
u_t+\kappa(u \cdot \nabla)u+\nabla P=\Delta u+n\nabla \phi,\quad
x\in \Omega, t>0,\\
\nabla\cdot u=0,\quad
x\in \Omega, t>0,\\
 \end{array}\right.\eqno(KSNF)
 $$
 which describes chemotaxis-fluid interaction in cases when the evolution of the chemoattractant is essentially
dominated by production through cells, where  $\kappa\in \mathbb{R},\phi\in W^{2,\infty}(\Omega)$ and $S$
is a given function with values in  $\mathbb{R}^{2\times2}$
which fulfills
$$|S(x,n,c)| \leq C_S (1 + n)^{-\alpha}$$
with some $C _S > 0$ and $\alpha \geq 0.$
If {\bf $\alpha>0$} and $\Omega\subseteq \mathbb{R}^2$ is a {\bf bounded} domain with
smooth boundary, then for all reasonably regular initial data, a corresponding initial-boundary value
problem for $(KSNF)$ possesses  a global classical solution which
is bounded on $\Omega\times(0,\infty)$. This extends
a recent result by Wang-Winkler-Xiang (Annali della Scuola Normale Superiore di Pisa-Classe di Scienze. XVIII, (2018), 2036--2145)
 which asserts global existence of bounded  solutions under the constraint
$\Omega\subseteq \mathbb{R}^2$ is a  bounded {\bf convex domain} with smooth boundary. Moreover,  we shall improve the
result of Wang-Xiang (J. Diff. Eqns., 259(2015), 7578--7609), who proved the possibility of global and bounded,
in the case that ${\bf\kappa\equiv0}$ and $\alpha>0$.
In comparison to the result
for the corresponding fluid-free system, the {\bf optimal condition} on the parameter $\alpha$ for both {\bf global existence} and {\bf boundedness}  are obtained.
Our main tool is consideration of the energy functional
$$\int_{\Omega}n^{1+\alpha} +\int_{\Omega}|\nabla c|^{2},$$
which is a new energy-like
functional.
\end{abstract}

\vspace{0.3cm}


\newpage
\section{Introduction}

\subsection{Chemotaxis model}
Many phenomena, which appear in natural science, especially, biology and
physics, support animals' lives (see \cite{Liggh00,Xu5566r793,Wangsseeess21215,Guggg1215}). The motion of cells moving towards the higher concentration of a chemical signal is called chemotaxis. For example, bacteria often swim toward higher concentrations of a signaling substance to survive.  In 1970, Chemotaxis model was first introduced by Keller and Segel (\cite{Keller2710,Keller79}), since then, a number of variations of modified chemotaxis models have been 
have been extensively studied (see Winkler et al. \cite{Bellomo1216},  Hillen-Painter \cite{Hillen} for example).
One particular class of
models is concerned with situations when the signal is  produced, by the
cells (see Tao-Winkler \cite{Tao794},  Winkler et al. \cite{Bellomo1216,Winkler792,Winkler793}). A correspondingly modified chemotaxis system, in its simplest form, is then given by
\begin{equation}
 \left\{\begin{array}{ll}
 n_t=\Delta n-\nabla\cdot( S(x,n,c)\nabla c),
 \quad
x\in \Omega,~ t>0,\\
 \disp{ c_t=\Delta c- c +n,}\quad
x\in \Omega, ~t>0,\\
 \end{array}\right.\label{ssd7223ddff44101.2x16677}
\end{equation}
where the tensor-valued function (or the scalar function)
 $S$ measures the chemotactic sensitivity, which may depend on the
cell density $n$  and chemosignal concentration  $c$, and also on the environmental variable $x$.
The results about the chemotaxis model \dref{ssd7223ddff44101.2x16677} appear to be rather complete.
In fact, when $S(x, n, c):= n(1+n)^{-\alpha}$ is a
scalar function, there is a critical exponent
\begin{equation}
 \alpha_{*}=1 -\frac{2}{N}\label{7223ddfff44101.2x1ddff6677}
\end{equation}
which is related to the presence of a so-called volume-filling
effect. More precisely, the solution of
system \dref{ssd7223ddff44101.2x16677} will globally exist  if $\alpha >\alpha_{*}$ (see Horstmann-Winkler \cite{Horstmann791} and Kowalczyk \cite{Kowalczyk7101}) and blow up in finite time
 if $\Omega$ is a ball in $\mathbb{R}^N (N \geq 2)$ and  $\alpha <\alpha_{*}$ under some technical assumptions (see Horstmann-Winkler \cite{Horstmann791} and Winkler \cite{Winkler79}).
 The global existence and boundedness of classical solution
to \dref{ssd7223ddff44101.2x16677} with the logistic source and more general form have also been investigated, see Cie\'{s}lak et al. \cite{Cie794,Cie791,Cie201712791}, Tao-Winkler \cite{Tao794,Winkler79,Winkler72} and Zheng et al. \cite{Zheng00,Zheng33312186,Zhengssssdefr23,Zhengssdefr23,Zhengsddfffsdddssddddkkllssssssssdefr23,Zhengssssssdefr23} and
the references therein.

\subsection{Chemotaxis-Navier-Stokes system}
In various situations, however, the migration of bacteria is furthermore substantially affected by
changes in their environment. For instance,
striking experimental evidence, as reported in \cite{Tuval1215} (see also Zhang-Zheng \cite{Zhang12176}, Winkler \cite{Winkler31215,Winkler61215,Winkler11215} and Wang-Winkler-Xiang \cite{Wang23421215}), reveals dynamical generation of patterns and
spontaneous emergence of turbulence in populations of aerobic bacteria suspended in sessile drops of
water. Taking into account all these processes, to describe the above biological phenomena, Tuval et al. (\cite{Tuval1215}) proposed the model
 the following chemotaxis-Navier-Stokes system
\begin{equation}
 \left\{\begin{array}{ll}
   n_t+u\cdot\nabla n=\Delta n-\nabla\cdot( nS(c)\nabla c),\quad
x\in \Omega, t>0,\\
    c_t+u\cdot\nabla c=\Delta c-nf(c),\quad
x\in \Omega, t>0,\\
u_t+\kappa (u\cdot\nabla)u+\nabla P=\Delta u+n\nabla \phi,\quad
x\in \Omega, t>0,\\
\nabla\cdot u=0,\quad
x\in \Omega, t>0\\
 \end{array}\right.\label{1.ss1hhjjddssggddddffftyy}
\end{equation}
for the unknown population density $n$, the signal concentration $c$, the fluid velocity field $u$ and the associated pressure $P$, in a bounded domain $\Omega\subseteq \mathbb{R}^N(N=2,3)$. Here $\phi$,  $\kappa\in \mathbb{R}$ and  $f(c)$  denote, respectively,  the  given potential function as well as the
strength of nonlinear fluid convection and the consumption rate of the oxygen by the bacteria.
 This model \dref{1.ss1hhjjddssggddddffftyy} describes a biological process in which cells  move towards a chemically more favorable environment.
  By making use of energy-type functionals, many literatures deal with
global solvability, boundedness, large time behavior of solutions to the model
\dref{1.ss1hhjjddssggddddffftyy} for the
bounded domains and the whole space (see e.g.
Lorz et al. \cite{Duan12186,Lorz1215},
Winkler et al. \cite{Bellomo1216,Winkler31215,Winkler51215}, Chae et al. \cite{Chaex12176,Chaexdd12176}, Di Francesco-Lorz-Markowich \cite{Francesco12186},
Zhang-Zheng \cite{Zhang12176} and references therein). For  example, Winkler (see Winkler \cite{Winkler31215,Winkler61215}) proved that in two-dimensional space \dref{1.ss1hhjjddssggddddffftyy}
admits a unique global classical solution which stabilizes to the spatially homogeneous equilibrium $(\frac{1}{|\Omega|}\int_{\Omega}n_0, 0, 0)$ in the large time limit. While in three-dimensional
setting, Winkler (see Winkler  \cite{Winkler51215}) also shows  that there exists a globally defined weak solution to \dref{1.ss1hhjjddssggddddffftyy}.
For the model with nonlinear diffusion, there also exist some results on global existence, boundedness and large time behavior for the bounded domains and the whole space
(see Duan-Xiang \cite{Duanx41215}, Di Francesco-Lorz-Markowich \cite{Francesco12186}, Ishida \cite{Ishida1215}, Winkler et al. \cite{Tao71215,Tao61215,Winklerdddsss51215}   and the references therein for details).

\subsection{Keller-Segel-Navier-Stokes system}

As pointed by Xue-Othmer (\cite{Xue1215}), the
environment for the bacterial cells is more complicated and other external forces
have to be considered. Hence, the chemotactic
sensitivity $S$ in system \dref{1.ss1hhjjddssggddddffftyy} should be replaced by the matrix $S(x,n,c)$
(see Xue-Othmer \cite{Xue1215}, Winkler \cite{Winklesssssrdddsss51215,Winkler11215},   Xue \cite{Xuess1215} and also Painter-Maini-Othmer \cite{Paintsser55677}), \dref{1.ss1hhjjddssggddddffftyy}
turns into a chemotaxis--Stokes system with rotational
flux. Thus, a generalization of the chemotaxis-fluid system \dref{1.ss1hhjjddssggddddffftyy}
should be of the form
\begin{equation}
 \left\{\begin{array}{ll}
   n_t+u\cdot\nabla n=\Delta n-\nabla\cdot( nS(x,n,c)\nabla c),\quad
x\in \Omega, t>0,\\
    c_t+u\cdot\nabla c=\Delta c-nf(c),\quad
x\in \Omega, t>0,\\
u_t+\kappa (u\cdot\nabla)u+\nabla P=\Delta u+n\nabla \phi,\quad
x\in \Omega, t>0,\\
\nabla\cdot u=0,\quad
x\in \Omega, t>0.\\
 \end{array}\right.\label{1.ss1hhjjsdfffddssggddddffftyy}
\end{equation}
In contrast to the chemotaxis-fluid system
\dref{1.ss1hhjjddssggddddffftyy}, chemotaxis-fluid systems \dref{1.ss1hhjjsdfffddssggddddffftyy} with tensor-valued sensitivity lose some natural gradient-like structure (see Winkler \cite{Winklesssssrdddsss51215,Winkler11215},  Wang-Winkler-Xiang \cite{Wang23421215} and also  Wang-Xiang \cite{Wang21215,Wangss21215}, Ke-Zheng \cite{Zhengssdddd00}). Thus, only very few results have appeared to be available on
chemotaxis-Stokes system \dref{1.ss1hhjjsdfffddssggddddffftyy} with such tensor-valued sensitivities so far (see e.g.  Wang-Cao \cite{Wang11215}, Winkler \cite{Winklesssssrdddsss51215,Winkler11215}).   For example, in the three dimensional case,
Wang and Cao (\cite{Wang11215}) obtained the global existence of classical solutions to system \dref{1.ss1hhjjsdfffddssggddddffftyy} with $\kappa = 0$
 and some decay on $S$.

 Concerning the framework where the chemical is {\bf produced} by the cells instead of consumed, then corresponding  chemotaxis--fluid model
 is then the quasilinear  Keller-Segel-Navier-Stokes system of the form
 \begin{equation}
 \left\{\begin{array}{ll}
   n_t+u\cdot\nabla n=\Delta n-\nabla\cdot(nS(x, n, c)\nabla c),\quad
x\in \Omega, t>0,\\
    c_t+u\cdot\nabla c=\Delta c-c+n,\quad
x\in \Omega, t>0,\\
u_t+\kappa (u\cdot\nabla)u+\nabla P=\Delta u+n\nabla \phi,\quad
x\in \Omega, t>0,\\
\nabla\cdot u=0,\quad
x\in \Omega, t>0,\\
 \disp{\left(nS(x, n, c)\nabla c\right)\cdot\nu=\nabla c\cdot\nu=0,u=0,}\quad
x\in \partial\Omega, t>0,\\
\disp{n(x,0)=n_0(x),c(x,0)=c_0(x),u(x,0)=u_0(x),}\quad
x\in \Omega\\
 \end{array}\right.\label{1.1}
\end{equation}
in a bounded domain in $\mathbb{R}^N$ with smooth boundary, where our main focus will be on the planar case {\bf $N = 2.$}
Here
$S(x, n, c)$ is a chemotactic sensitivity tensor satisfying
\begin{equation}\label{x1.73142vghf48rtgyhu}
S\in C^2(\bar{\Omega}\times[0,\infty)^2;\mathbb{R}^{2\times2})
 \end{equation}
 and
 \begin{equation}\label{x1.73142vghf48gg}|S(x, n, c)|\leq C_S(1 + n)^{-\alpha} ~~~~\mbox{for all}~~ (x, n, c)\in\Omega\times [0,\infty)^2
 \end{equation}
with some $C_S > 0$ and $\alpha\geq 0$.
Such chemotaxis
fluid system with signal production arises in the modeling of bacterial populations living in the liquid (\cite{Bellomo1216,Hillen}).

Due to the presence
of the tensor-valued sensitivity as well as the strongly nonlinear term $(u \cdot \nabla)u$ and lower  regularity for $n$ ($\int_{\Omega}n= \int_{\Omega}n_0$, as the only apparent a priori information available),
the mathematical analysis of \dref{1.1} regarding global and bounded solutions is far from
trivial
 (see Wang-Xiang et. al.  \cite{Peng55667,Wang21215,Wangss21215,Wangssddss21215}, Zheng \cite{Zhenddddgssddsddfff00}). In fact, in 2-dimensional, if $S=S(x, n, c)$ is a tensor-valued sensitivity fulfilling \dref{x1.73142vghf48rtgyhu}
and \dref{x1.73142vghf48gg}, Wang and Xiang (\cite{Wang21215}) proved that Stokes-version ($\kappa=0$ in the first equation of \dref{1.1}) of system \dref{1.1}
admits a unique global classical solution which is bounded. Then Wang-Winkler-Xiang (\cite{Wang23421215}) further shows that when $\alpha >0$ and $\Omega\subseteq \mathbb{R}^2$ is a bounded {\bf convex domain} with smooth boundary,
system \dref{1.1} possesses a global-in-time classical and bounded solution.
 We should pointed that the key approach of \cite{Wang23421215} is to establish the boundedness involves an analysis of the functional
\begin{equation}\int_{\Omega}\ln  n(\cdot,t)n(\cdot,t) +a\int_{\Omega} |\nabla {c}(\cdot,t)|^{2}
 \label{gghhggghnssdjjffff1.1hhjjddssggtyy}
\end{equation}
and some bootstrap argument, where   $a > 0$ is a suitable positive constant, $n$ and $c$ are components of the solutions to  \dref{1.1} (see
Lemma 6.2 of \cite{Wang23421215}).
Combined with \dref{gghhggghnssdjjffff1.1hhjjddssggtyy}, in view of the {\bf convexity} of $\Omega$, one can  obtain the upper bound of the functional
\begin{equation}\frac{1}{p}\int_{\Omega}n^{p}(\cdot,t) +\frac{2}{q}\int_{\Omega} |\nabla {c}(\cdot,t)|^{2{q}}
 \label{gghhggghnjddddjffff1.1hhjjddssggtyy}
\end{equation}
 (see  Lemma 7.1 of \cite{Wang23421215}). To the best
of our knowledge, it is yet unclear whether for $\kappa\neq0$ and $\Omega\subseteq \mathbb{R}^2$ is a {\bf non-convex} domain, the solution of \dref{1.1} is
is bounded  or not. {\bf By using a different method} involving more general entropy-like functionals
(see Lemma \ref{lemma4556664ddd5630223}), the present paper will extend the above result so as to cover the {\bf non-convex domain}.
In a three-dimensional setup  and tensor-valued sensitivity
$S$ satisfying \dref{x1.73142vghf48gg} global weak solutions have been shown to exists
for $\alpha > \frac{3}{7}$ (see \cite{LiuZhLiuLiuandddgddff4556}) and $\alpha > \frac{1}{3}$  (see \cite{Zhengssdddd00} and also \cite{Wangssddss21215}), respectively.


Throughout this paper,
we assume that
\begin{equation}
\phi\in W^{2,\infty}(\Omega)
\label{dd1.1fghyuisdakkkllljjjkk}
\end{equation}
 and the initial data
$(n_0, c_0, u_0)$ fulfills
\begin{equation}\label{ccvvx1.731426677gg}
\left\{
\begin{array}{ll}
\displaystyle{n_0\in C^\kappa(\bar{\Omega})~~\mbox{for certain}~~ \kappa > 0~~ \mbox{with}~~ n_0\geq0 ~~\mbox{in}~~\Omega},\\
\displaystyle{c_0\in W^{2,\infty}(\Omega)~~\mbox{with}~~c_0,w_0\geq0~~\mbox{in}~~\bar{\Omega},}\\
\displaystyle{u_0\in D(A),}\\
\end{array}
\right.
\end{equation}
where $A$ denotes the Stokes operator with domain $D(A) := W^{2,{2}}(\Omega)\cap  W^{1,{2}}_0(\Omega)
\cap L^{2}_{\sigma}(\Omega)$,
and
$L^{2}_{\sigma}(\Omega) := \{\varphi\in  L^{2}(\Omega)|\nabla\cdot\varphi = 0\}$. 
 (\cite{Sohr}).

Within the above frameworks, our main result concerning global existence and boundedness
of solutions to \dref{1.1} is as follows.
\begin{theorem}\label{theorem3}
Let  $\Omega\subset \mathbb{R}^2$ be a bounded    domain with smooth boundary,
 \dref{dd1.1fghyuisdakkkllljjjkk} and \dref{ccvvx1.731426677gg}
 hold. Moreover, assume that suppose that $S$ satisfies \dref{x1.73142vghf48rtgyhu}
and \dref{x1.73142vghf48gg}
with some
\begin{equation}\label{x1.73142vghf48}\alpha>0.
\end{equation}
Then for any choice of $n_0, c_0$ and $u_0$ fulfilling \dref{ccvvx1.731426677gg},
 the problem \dref{1.1}  possesses a global classical solution $(n, c, u, P)$
 which satisfies
 \begin{equation}
 \left\{\begin{array}{ll}
 n\in C^0(\bar{\Omega}\times[0,\infty))\cap C^{2,1}(\bar{\Omega}\times(0,\infty)),\\
  c\in  C^0(\bar{\Omega}\times[0,\infty))\cap C^{2,1}(\bar{\Omega}\times(0,\infty))\cap L^\infty([0,\infty); W^{1,p}(\Omega))~~\mbox{with}~~p>1,\\
  u\in  C^0(\bar{\Omega}\times[0,\infty);\mathbb{R}^2)\cap C^{2,1}(\bar{\Omega}\times(0,\infty);\mathbb{R}^2)\cap L^\infty([0,\infty); D(A^\gamma))~~\mbox{with}~~\gamma\in(0,1),\\
  P\in  C^{1,0}(\bar{\Omega}\times(0,\infty))\\
   \end{array}\right.\label{1.ffhhh1ddfghyuisda}
\end{equation}
as well as $n$ and $c$ are nonnegative in
$\Omega\times(0,\infty)$.
Moreover, this solution is bounded in the sense that for each $p > 1$
and any $\gamma\in (0, 1)$, there exists $C(p, \gamma) > 0$ with the property that
\begin{equation}
\|n(\cdot, t)\|_{L^\infty(\Omega)}+\|c(\cdot, t)\|_{W^{1,p}(\Omega)}+\| A^\gamma u(\cdot, t)\|_{L^{2}(\Omega)}\leq C(p, \gamma)~~ \mbox{for all}~~ t>0.
\label{1.163072xggttyyu}
\end{equation}
\end{theorem}
\begin{remark}
(i)  We remove the {\bf convexity} of $\Omega$ required in \cite{Wang23421215}.


 (ii)
 If $u\equiv0$,  Theorem \ref{theorem3} is (partly) coincides with
Theorem 4.1 of \cite{Winkler79}, which is {\bf optimal} according to
the fact that the 2D fluid-free system  admits a global bounded classical solution  for $\alpha>0$
 as
mentioned before.

 (iii) Theorem \ref{theorem3}  extends the results of    Wang-Xiang \cite{Wang21215},
who proved the possibility of boundedness,
in the case that {\bf $\kappa=0$} and  $S$ satisfies \dref{x1.73142vghf48rtgyhu}
as well as  \dref{x1.73142vghf48gg}
with some $\alpha>0.$

(iv)
In the two dimensional case, if $\kappa=0$ in \dref{1.1},
Li-Wang-Xiang \cite{Liggghh793} showed the global existence and boundedness of weak solutions to
system \dref{1.1} with  non-linear diffusion ($\Delta n$ is replaced by $\Delta n^m$ in the first equation of \dref{1.1}) for any $m > 1$. Using the idea and method of this paper, one can extend the above result so as to cover the case $m>1$ and {\bf $\kappa\in \mathbb{R}$} (see our recent paper \cite{Zhenssddddgssddsddfff00}), which is also optimal according to \dref{7223ddfff44101.2x1ddff6677}.

\end{remark}

We sketch here the main ideas and methods used in this article. Our approach underlying the derivation of Theorem \ref{theorem3}
will be based on an entropy-like estimate involving the functional
\begin{equation}
\int_{\Omega}n_\varepsilon^{1+\alpha} +\int_{\Omega}|\nabla c_{\varepsilon}|^{2}
\label{334444zjscz2ddfff.52ddffddffff97x9630222ssdd2114}
\end{equation}
for solutions of certain regularized versions of \dref{1.1} (see Lemma \ref{lemma4556664ddd5630223}),  which is a new estimate
of Keller-Segel-Navier-Stokes system with rotation.  Once this crucial
step has been accomplished, one can derive the global boundedness of solution to \dref{1.1}
by an application of well-known  regularization estimates for the parabolic  theory and  Stokes semigroup (see Lemma \ref{lemma45630hhuujj}).

This paper is organized as follows.  In Section 2, we introduce the regularized system of \dref{1.1}, show
 the local existence of solutions to the regularized system of
 \dref{1.1} and give two
preliminary lemmas.  After collecting some basic estimates of the solutions in Section 3, we prove global  existence of  solution to  regularized problems of \dref{1.1} in  Section 4. On the basis of the compactness properties thereby implied, in Section 5 and Section 6 we shall finally
pass to the limit along an adequate sequence of numbers $\varepsilon = \varepsilon_j\searrow0$
and thereby verify the main results.

\section{Preliminaries}

In this section, we give some notations and recall some basic facts which will be
frequently used throughout the paper.
In order to construct solutions of \dref{1.1} through an appropriate approximation,
%
 let us fix families
$(\rho_\varepsilon)_{\varepsilon\in(0,1)} $ and $(\chi_\varepsilon)_{\varepsilon\in(0,1)} $
of functions
$$\rho_\varepsilon \in C^\infty_0 (\Omega)~~\mbox{such that}~~0\leq\rho_\varepsilon\leq1~~\mbox{in}~~\Omega~~\mbox{and}~~\rho_\varepsilon\nearrow1~~\mbox{in}~~\Omega~~\mbox{as}~~\varepsilon\searrow0$$
and
$$\chi_\varepsilon \in C^\infty_0 ([0,\infty))~~\mbox{such that}~~0\leq\chi_\varepsilon\leq1~~\mbox{in}~~[0,\infty)~~\mbox{and}~~\chi_\varepsilon\nearrow1~~\mbox{in}~~[0,\infty)~~\mbox{as}~~\varepsilon\searrow0,$$
define
\begin{equation}
\begin{array}{ll}
S_\varepsilon(x, n, c) := \rho_\varepsilon(x)\chi_\varepsilon(u)S(x, n, c),~~ x\in\bar{\Omega},~~n\geq0,~~c\geq0
 \end{array}\label{3.10gghhjuuloollyuigghhhyy}
\end{equation}
and consider the following the approximate system
\begin{equation}
 \left\{\begin{array}{ll}
   n_{\varepsilon t}+u_{\varepsilon}\cdot\nabla n_{\varepsilon}=\Delta n_{\varepsilon}-\nabla\cdot(n_{\varepsilon}S_\varepsilon(x, n_{\varepsilon}, c_{\varepsilon})\nabla c_{\varepsilon}),\quad
x\in \Omega,\; t>0,\\
    c_{\varepsilon t}+u_{\varepsilon}\cdot\nabla c_{\varepsilon}=\Delta c_{\varepsilon}-c_{\varepsilon}+n_{\varepsilon},\quad
x\in \Omega,\; t>0,\\
u_{\varepsilon t}+\nabla P_{\varepsilon}=\Delta u_{\varepsilon}-\kappa (Y_{\varepsilon}u_{\varepsilon} \cdot \nabla)u_{\varepsilon}+n_{\varepsilon}\nabla \phi,\quad
x\in \Omega,\; t>0,\\
\nabla\cdot u_{\varepsilon}=0,\quad
x\in \Omega,\; t>0,\\
 \disp{\nabla n_{\varepsilon}\cdot\nu=\nabla c_{\varepsilon}\cdot\nu=0,u_{\varepsilon}=0,\quad
x\in \partial\Omega,\; t>0,}\\
\disp{n_{\varepsilon}(x,0)=n_0(x),c_{\varepsilon}(x,0)=c_0(x),\;u_{\varepsilon}(x,0)=u_0(x)},\quad
x\in \Omega,\\
 \end{array}\right.\label{1.1fghyuisda}
\end{equation}
where
\begin{equation}
 \begin{array}{ll}
 Y_{\varepsilon}w := (1 + \varepsilon A)^{-1}w ~\mbox{for all}~ w\in L^2_{\sigma}(\Omega)
 \end{array}\label{aasddffgg1.1fghyuisda}
\end{equation}
is a standard Yosida approximation.

Let us recall a result on local solvability of \dref{1.1fghyuisda}, which has been established in \cite{Winkler31215} (see also Bellomo et al. \cite{Bellomo1216}) by
means of a suitable
extensibility criterion and a slight modification of the well-established fixed-point arguments (see Lemma 2.1 of \cite{Winkler51215}, \cite{Winkler11215} and Lemma 2.1 of \cite{Painter55677}).

\begin{lemma}\label{lemma70}
Let $\Omega \subseteq \mathbb{R}^2$ be a bounded domain with smooth boundary.
Assume
that $\phi\in W^{1,\infty}(\Omega)$ and the initial data $(n_0,c_0,u_0)$ fulfills \dref{ccvvx1.731426677gg}.
Moreover, let
  $S\in C^2(\bar{\Omega}\times[0,\infty)^2;\mathbb{R}^{2\times2})$
satisfy  \dref{x1.73142vghf48gg} for some $C_S \geq 0$ and $\alpha\geq 0$.
Then for $\varepsilon\in(0,1),$ there exist $T_{max,\varepsilon}\in  (0,\infty]$ and
a classical solution $(n_\varepsilon, c_\varepsilon, u_\varepsilon, P_\varepsilon)$ of \dref{1.1fghyuisda} in
$\Omega\times(0, T_{max,\varepsilon})$ such that
\begin{equation}
 \left\{\begin{array}{ll}
 n_\varepsilon\in C^0(\bar{\Omega}\times[0,T_{max,\varepsilon}))\cap C^{2,1}(\bar{\Omega}\times(0,T_{max,\varepsilon})),\\
  c_\varepsilon\in  C^0(\bar{\Omega}\times[0,T_{max,\varepsilon}))\cap C^{2,1}(\bar{\Omega}\times(0,T_{max,\varepsilon}))\cap\bigcap_{p>1} L^\infty([0,T_{max,\varepsilon}); W^{1,p}(\Omega)),\\
  u_\varepsilon\in  C^0(\bar{\Omega}\times[0,T_{max,\varepsilon}))\cap C^{2,1}(\bar{\Omega}\times(0,T_{max,\varepsilon}))\cap \bigcap_{\gamma\in(0,1)}C^0([0,T_{max,\varepsilon}); D(A^\gamma)),\\
  P_\varepsilon\in  C^{1,0}(\bar{\Omega}\times(0,T_{max,\varepsilon}))\\
   \end{array}\right.\label{1.1ddfghyuisda}
\end{equation}
 classically solving \dref{1.1fghyuisda} in $\Omega\times[0,T_{max,\varepsilon})$.
%
Moreover,  $n_\varepsilon$ and $c_\varepsilon$ are nonnegative in
$\Omega\times(0, T_{max,\varepsilon})$, and
that
\begin{equation}
\mbox{if}~~T_{max,\varepsilon}<+\infty,~~\mbox{then}~~\limsup_{t\nearrow T_{max,\varepsilon}}[\|n_\varepsilon(\cdot, t)\|_{L^\infty(\Omega)}+\|c_\varepsilon(\cdot, t)\|_{W^{1,\infty}(\Omega)}+\|A^\gamma u_\varepsilon(\cdot, t)\|_{L^{2}(\Omega)}]=\infty
\label{1.163072x}
\end{equation}
for all $p > 2$ and $\gamma\in ( \frac{1}{2}, 1).$
\end{lemma}

\begin{lemma}(\cite{Horstmann791})\label{llssdrffmmggnnccvvccvvkkkkgghhkkllvvlemma45630}
If $m\in \{0, 1\}, p\in [1,\infty]$ and
$q \in (1,\infty)$ then with some constant $c_1 > 0$, for all $w\in D((-\Delta+1)^\theta)$ we have
$$
\|\varphi\|_{W^{m,p}(\Omega)}\leq c\|(-\Delta+1)^\theta \varphi\|_{L^{q}(\Omega)}~~~\mbox{provided that}~~m-\frac{N}{p} < 2\theta-\frac{N}{q}.
$$

Moreover, for $p < \infty$ the associated heat semigroup$(e^{t\Delta})_{t\geq0}$ maps $L^p(\Omega)$ into $D((-\Delta+1)^\theta)$ in any of the
spaces $L^q(\Omega)$ for $q\geq p$, and there exist $c > 0$ and $\lambda > 0$ such that the $L^p$-$L^q$ estimates
$$
\|(-\Delta+1)^\theta e^{t(\Delta-1)}\varphi\|_{L^q(\Omega)} \leq c(1+t^{-\theta-\frac{N}{2}(\frac{1}{p}-\frac{1}{q})})e^{-\lambda t}\|\varphi\|_{L^p(\Omega)} ~\mbox{for all}~\varphi\in L^{p}(\Omega)~~\mbox{and}~~t>0
$$
and
$$
\begin{array}{rl}
&\|(-\Delta+1)^\theta e^{t\Delta}\varphi\|_{L^q(\Omega)}\\
 \leq &c(1+t^{-\theta-\frac{N}{2}(\frac{1}{p}-\frac{1}{q})})e^{-\lambda t}\|\varphi\|_{L^p(\Omega)} ~~\mbox{for all}~~t>0~~\mbox{and}~~\varphi\in L^{p}(\Omega)~~\mbox{satisfying}~~\disp\int_{\Omega} w = 0.
\end{array}
$$
In addition,  given $p\in (1,\infty)$, for any $\varepsilon > 0$, there exists $c(\varepsilon) > 0$ such that for all $\mathbb{R}^N$-valued $\varphi\in L^p(\Omega),$
$$
\begin{array}{rl}
\|(-\Delta+1)^\theta e^{t\Delta}\nabla\cdot\varphi\|_{L^p(\Omega)} \leq &c(\varepsilon)(1+t^{-\theta-\frac{1}{2}-\varepsilon})e^{-\lambda t}\|\varphi\|_{L^p(\Omega)} ~~\mbox{for all}~~t>0.
\end{array}
$$
\end{lemma}

\begin{lemma}(\cite{Tao41215})\label{lemma630}
Let $T\in(0,\infty]$, 
$\sigma\in(0,T)$, $A>0$ and $B>0$, and suppose that
$y:[0,T)\rightarrow[0,\infty)$ is absolutely
continuous and such that
\begin{equation}\label{x1.73142hjkl}
\begin{array}{ll}
\displaystyle{
 y'(t)+Ay(t)\leq h(t) ~~\mbox{for a.e.}~~t\in(0,T)}\\
\end{array}
\end{equation}
with some nonnegative function $h\in  L^1_{loc}([0, T))$ satisfying
$$\int_{t}^{t+\sigma}h(s)ds\leq B~~\mbox{for all}~~t\in(0,T-\sigma).$$
Then
$$y(t)\leq \max\{y_0+B,\frac{B}{A\tau}+2B\}~~\mbox{for all}~~t\in(0,T).$$
\end{lemma}

\section{Some basic priori estimates}

In this section, in order to  establish the global solvability of system \dref{1.1fghyuisda}, we proceed to derive $\varepsilon$-independent estimates for the approximate solutions constructed above. As the first step, we need to establish some important a priori estimates for $n_\varepsilon, c_\varepsilon$
 and $u_\varepsilon$, where throughout
this paper, $(n_\varepsilon, c_\varepsilon, u_\varepsilon, P_\varepsilon)$ is the global solution of problem  \dref{1.1fghyuisda}.

The following basic properties of solutions to \dref{1.1fghyuisda} are immediate.
\begin{lemma}\label{fvfgfflemma45}
The solution of \dref{1.1fghyuisda} satisfies
\begin{equation}
\int_{\Omega}{n_{\varepsilon}}= \int_{\Omega}{n_{0}}~~\mbox{for all}~~ t\in(0, T_{max,\varepsilon})
\label{ddfgczhhhh2.5ghju48cfg924ghyuji}
\end{equation}
as well as
\begin{equation}
\int_{\Omega}{n_{\varepsilon}}\leq \max\{\int_{\Omega}{n_{0}},\int_{\Omega}{c_{0}}\}~~\mbox{for all}~~ t\in(0, T_{max,\varepsilon}).
\label{2344ddfgczhhhh2.5ghju48cfg924ghyuji}
\end{equation}
%
%
\end{lemma}

With the above Lemma at hand, a series of straightforward integrations by parts will lead to the following
energy-type equality which,  was already used in Lemma 3.3 in \cite{Zhengssdddd00} (see also \cite{Zhenddddgssddsddfff00,Wang23421215}).

\begin{lemma}\label{lemmajddggmk43025xxhjklojjkkk}
Let $\alpha>0$.
Then there exists $C>0$ independent of $\varepsilon$ such that the solution of \dref{1.1fghyuisda} satisfies
\begin{equation}
\begin{array}{rl}
&\disp{\int_{\Omega}n_{\varepsilon}+\int_{\Omega} n_{\varepsilon}^{2\alpha}+\int_{\Omega}   c_{\varepsilon}^2+\int_{\Omega}  | {u_{\varepsilon}}|^2\leq C~~~\mbox{for all}~~ t\in (0, T_{max,\varepsilon}).}\\
\end{array}
\label{czfvgb2.5ghhjuyuccvviihjj}
\end{equation}
Moreover, for all $t\in(0, T_{max,\varepsilon}-\tau)$,
it holds that
one can find a constant $C > 0$ independent of $\varepsilon$ such that
\begin{equation}
\begin{array}{rl}
&\disp{\int_{t}^{t+\tau}\int_{\Omega} \left[  n_{\varepsilon}^{2\alpha-2} |\nabla {n_{\varepsilon}}|^2+ |\nabla {c_{\varepsilon}}|^2+ |\nabla {u_{\varepsilon}}|^2\right]\leq C,}\\
\end{array}
\label{bnmbncz2.5ghhjuyuivvbnnihjj}
\end{equation}
where $\tau=\min\{1,\frac{1}{6}T_{max,\varepsilon}\}.$
\end{lemma}




In order to extract helpful boundedness information concerning $n_{\varepsilon}$, we first require
estimates for higher norms of $c_{\varepsilon}$ on the basis of Lemma \ref{lemmajddggmk43025xxhjklojjkkk}.

\begin{lemma}\label{aasslemmafggg78630jklhhjj}
 Let  $(n_\varepsilon,c_\varepsilon,u_\varepsilon)$ be the solution of \dref{1.1ddfghyuisda} and $\tau=\min\{1,\frac{1}{6}T_{max,\varepsilon}\}.$
If
 there  exists $K$ 
\begin{equation}\int_{t}^{t+\tau}\left(\| \nabla{ n_{\varepsilon}^{\alpha}}\|^{2}_{L^{2}(\Omega)}\right)ds\leq K~~ \mbox{for all}~~ t\in(0,T_{max,\varepsilon}-\tau),
\label{3.10gghhjuuloollgghhhy}
\end{equation}
then for any
  $q>2$ 
then there exists $C: = C(q,K)$ independent of $\varepsilon$ such that
 \begin{equation}\|c_{\varepsilon}(\cdot, t)\|_{L^q(\Omega)}\leq C~~ \mbox{for all}~~ t\in(0,T_{max,\varepsilon}).
\label{3.10gghhjuuloollgghhhyhh}
\end{equation}
\end{lemma}
\begin{proof}
For all $p>\max\{1+4\alpha,2\}$,
testing  the second equation of
$\dref{1.1fghyuisda}$ by ${c^{p-1}_{\varepsilon}}$
 and combining with the second equation and using $\nabla\cdot u_{\varepsilon}=0$, we have, using the integration by parts, that
\begin{equation}
\begin{array}{rl}
&\disp{\frac{1}{p}\frac{d}{dt}\int_{\Omega}c^{{{p}}}_{\varepsilon}+({{p}-1})\int_{\Omega}c^{{{p}-2}}_{\varepsilon}|\nabla c_{\varepsilon}|^2+\int_{\Omega}c^{{{p}}}_{\varepsilon}}\\
=&\disp{\int_\Omega c^{p-1}_{\varepsilon}n_{\varepsilon}}\\
\leq&\disp{\|n_{\varepsilon}\|_{L^\frac{p-2\alpha}{p-4\alpha}(\Omega)}\left(\int_\Omega c^{\frac{(p-1)(p-2\alpha)}{{2\alpha}}}_{\varepsilon}\right)^{\frac{{2\alpha}}{p-2\alpha}}~~\mbox{for all}~~t\in(0,T_{max,\varepsilon})}\\
\end{array}
\label{3333cz2.5114114}
\end{equation}
by the H\"{o}lder inequality.
On the other hand, by the Gagliardo--Nirenberg inequality and \dref{ddfgczhhhh2.5ghju48cfg924ghyuji}, one can get there exist positive constants  $\mu_0$ and $\mu_1$ such that
\begin{equation}
\begin{array}{rl}
\disp\left(\int_\Omega c^{\frac{(p-1)(p-2\alpha)}{{2\alpha}}}_{\varepsilon}\right)^{\frac{{2\alpha}}{p-2\alpha}}=
&\disp{\|  c^{\frac{p}{2}}_{\varepsilon}\|^{\frac{2(p-1)}{p}}_{L^{\frac{(p-1)(p-2\alpha)}{p\alpha }}(\Omega)}}\\
\leq&\disp{\mu_{0}\|\nabla   c^{\frac{p}{2}}_{\varepsilon}\|_{L^2(\Omega)}^{\frac{2(p-2\alpha-1)}{p-2\alpha}}\|  c^{\frac{p}{2}}_{\varepsilon}\|_{L^{\frac{2}{p}}(\Omega)}^{\frac{{4\alpha}}{p(p-2\alpha)}}+\|  c^{\frac{p}{2}}_{\varepsilon}\|_{L^\frac{2}{p}(\Omega)}^{\frac{2(p-1)}{p}}}\\
\leq&\disp{\mu_{1}(\|\nabla   c^{\frac{p}{2}}_{\varepsilon}\|_{L^2(\Omega)}^{\frac{2(p-2\alpha-1)}{p-2\alpha}}+1).}\\
\end{array}
\label{123cz2.57151hhddfsdffffkkhhhjddffffgukildrftjj}
\end{equation}
Inserting \dref{123cz2.57151hhddfsdffffkkhhhjddffffgukildrftjj} into \dref{3333cz2.5114114} and using the Young inequality, we derive that
\begin{equation}
\begin{array}{rl}
&\disp{\frac{1}{p}\frac{d}{dt}\int_{\Omega}c^{{{p}}}_{\varepsilon}+({{p}-1})\int_{\Omega}c^{{{p}-2}}_{\varepsilon}|\nabla c_{\varepsilon}|^2+\int_{\Omega}c^{{{p}}}_{\varepsilon}}\\
\leq&\disp{\mu_{1}\|n_{\varepsilon}\|_{L^\frac{p-2\alpha}{p-4\alpha}(\Omega)}(\|\nabla   c^{\frac{p}{2}}_{\varepsilon}\|_{L^2(\Omega)}^{\frac{2(p-2\alpha-1)}{p-2\alpha}}+1)}\\
\leq&\disp{\frac{({{p}-1})}{2}\int_{\Omega}c^{{{p}-2}}_{\varepsilon}|\nabla c_{\varepsilon}|^2+C_1(p)\mu_{1}^{p-2\alpha}\|n_{\varepsilon}\|_{L^\frac{p-2\alpha}{p-4\alpha}(\Omega)}^{p-2\alpha}+\mu_{1}\|n_{\varepsilon}\|_{L^\frac{p-2\alpha}{p-4\alpha}(\Omega)}~\mbox{for all}~
t\in(0,T_{max,\varepsilon}).}\\
\end{array}
\label{ssddd3333cz2.51fggtyuujkkklii14114}
\end{equation}
In view of  $p>\max\{1+4\alpha,2\}$, we derive the Young inequality that
\begin{equation}
\begin{array}{rl}
&\disp{\frac{1}{p}\frac{d}{dt}\int_{\Omega}c^{{{p}}}_{\varepsilon}+({{p}-1})\int_{\Omega}c^{{{p}-2}}_{\varepsilon}|\nabla c_{\varepsilon}|^2+\int_{\Omega}c^{{{p}}}_{\varepsilon}}\\
\leq&\disp{\frac{({{p}-1})}{2}\int_{\Omega}c^{{{p}-2}}_{\varepsilon}|\nabla c_{\varepsilon}|^2+C_2(p)\mu_{1}^{p-2\alpha}\|n_{\varepsilon}\|_{L^\frac{p-2\alpha}{p-4\alpha}(\Omega)}^{p-2\alpha}+C_2(p)~\mbox{for all}~
t\in(0,T_{max,\varepsilon}).}\\
\end{array}
\label{3333cz2.51fggtyuujkkklii14114}
\end{equation}
To obtain the uniform bound of the above functional, we will employ it to bound the dissipation from
below. To this end, due to \dref{czfvgb2.5ghhjuyuccvviihjj}
and \dref{bnmbncz2.5ghhjuyuivvbnnihjj}, for some $C_3,C_4$ and $C_5> 0$ which are independent of $\varepsilon$,
 we use the Gagliardo-Nirenberg inequality 
to obtain
\begin{equation}
\begin{array}{rl}
&\disp\int_{t}^{t+\tau}\left[\|n_{\varepsilon}\|_{L^\frac{p-2\alpha}{p-4\alpha}(\Omega)}^{p-2\alpha}+C_2(p)\right]ds \\
=&\disp{\int_{t}^{t+\tau}\left[\|  n_{\varepsilon}^{\alpha}\|^{\frac{p-2\alpha}{2\alpha}}_{L^{\frac{p-2\alpha}{(p-4\alpha)\alpha }}(\Omega)}+C_2(p)\right]ds}\\
\leq&\disp{C_{3}\int_{t}^{t+\tau}\left(\| \nabla{ n_{\varepsilon}^{\alpha}}\|^{2}_{L^{2}(\Omega)}\|{ n_{\varepsilon}^{\alpha}}\|^{{\frac{p}{2\alpha}}}_{L^{\frac{1}{2\alpha}}(\Omega)}+
\|{ n_{\varepsilon}^{\alpha}}\|^{\frac{p-2\alpha}{2\alpha}}_{L^{\frac{1}{2\alpha}}(\Omega)}\right)ds+C_2(p)}\\
\leq&\disp{C_{4}\int_{t}^{t+\tau}\left(\| \nabla{ n_{\varepsilon}^{\alpha}}\|^{2}_{L^{2}(\Omega)}\right)ds+C_2(p)}\\
\leq&\disp{C_{5}K,}\\
\end{array}
\label{ddffbnmbnddfgcz2ddfvgbhh.htt678ddfghhhyuiihjj}
\end{equation}
where $\tau=\min\{1,\frac{1}{6}T_{max,\varepsilon}\}.$
This enables us to apply Lemma \ref{lemma630} to conclude that \dref{3.10gghhjuuloollgghhhyhh}  by the H\"{o}lder inequality. 
\end{proof}

With the estimates obtained so far,  we have already prepared all tools to obtain an $ L^p(\Omega)$-estimate for $n_{\varepsilon}$, for some $p >1$, which plays a key role
in obtaining the $ L^\infty(\Omega)$-estimate for $n_{\varepsilon}$.  In \cite{Wang23421215}, Wang, Winkler and Xiang proved the existence of a
global classical solutions in the 2D case  on the basis of the free-energy
inequality
\begin{equation}
\int_{\Omega}n_\varepsilon\ln n_\varepsilon +a\int_{\Omega}|\nabla c_{\varepsilon}|^{2} ~~~\mbox{for all}~~ t\in(0,T_{max,\varepsilon})~~\mbox{and any}~~\varepsilon>0
\label{334444zjscz2.52ddfff97x9630222ssdd2114}
\end{equation}
with some suitable $a > 0$.
 The novelty of the
present reasoning, 
%
we want to derive a  entropy-like functionals
%
\begin{equation}
\int_{\Omega}n_\varepsilon^p +\int_{\Omega}|\nabla c_{\varepsilon}|^{2} ~~\mbox{with some}~~p>1,
\label{334444zjscz2.52ddffddffff97x9630222ssdd2114}
\end{equation}
which is  different from \cite{Wang23421215}.

\begin{lemma}\label{lemma4556664ddd5630223}
If
  \begin{equation}\label{gddffffnjjmmx1.73ddddd1426677gg}
\alpha>0,
\begin{array}{ll}\\
 \end{array}
\end{equation}
then 
 the solution of \dref{1.1fghyuisda} from Lemma \ref{lemma70} satisfies
\begin{equation}
\int_{\Omega}n^{1+\alpha}_\varepsilon(x,t) \leq C ~~~\mbox{for all}~~ t\in(0,T_{max,\varepsilon})~~\mbox{and any}~~\varepsilon>0
\label{111334dddd444zjscz2.5297x9630222ssdd2114}
\end{equation}
as well as
\begin{equation}
\int_{\Omega}|\nabla c_{\varepsilon}|^{2} \leq C ~~~\mbox{for all}~~ t\in(0,T_{max,\varepsilon})~~\mbox{and any}~~\varepsilon>0
\label{334dddd444zjscz2.5297x9630222ssdd2114}
\end{equation}
and
\begin{equation}\int_{t}^{t+\tau}\int_{\Omega} n_\varepsilon  ^{2{+\alpha}}\leq C~~ \mbox{for all}~~ t\in(0,T_{max,\varepsilon}-\tau)~~\mbox{and any}~~\varepsilon>0,
\label{3.10gghhjuuloollsdffffffgghhhy}
\end{equation}
where $\tau=\min\{1,\frac{1}{6}T_{max,\varepsilon}\}.$
\end{lemma}

\begin{proof}
%
Taking ${ n_{\varepsilon}^{\alpha}}$ as the test function for the first equation of
$\dref{1.1fghyuisda}$
 and combining with the second equation and using $\nabla\cdot u_\varepsilon=0$, we derive 
 that

\begin{equation}
\begin{array}{rl}
&\disp{\frac{1}{{1+\alpha}}\frac{d}{dt}\|n_\varepsilon \|^{{1+\alpha}}_{L^{{1+\alpha}}(\Omega)}+{\alpha}\int_{\Omega} n_\varepsilon  ^{{{{\alpha}-1}}}|\nabla n_\varepsilon|^2}
\\
=&\disp{-\int_\Omega  n_\varepsilon  ^{\alpha}\nabla\cdot(n_\varepsilon S_\varepsilon(x, n_{\varepsilon}, c_{\varepsilon})
\nabla c_\varepsilon) }\\
=&\disp{ \alpha \int_\Omega   n_\varepsilon  ^{{\alpha-1}} n_\varepsilon S_\varepsilon(x, n_{\varepsilon}, c_{\varepsilon})
\nabla n_\varepsilon\cdot\nabla c_\varepsilon}\\
\leq&\disp{ \alpha C_S \int_\Omega   n_\varepsilon  ^{\alpha}(1+n_\varepsilon)^{-\alpha}
|\nabla n_\varepsilon||\nabla c_\varepsilon|~~\mbox{for all}~~ t\in(0,T_{max,\varepsilon})}\\
\end{array}
\label{ttty3333cz2.5114114}
\end{equation}
by using \dref{x1.73142vghf48gg}.
Therefore,
by the Young inequality, we conclude that
\begin{equation}
\begin{array}{rl}
&\disp{\frac{1}{{1+\alpha}}\frac{d}{dt}\|n_\varepsilon  \|^{{{1+\alpha}}}_{L^{{1+\alpha}}(\Omega)}+{\alpha}\int_{\Omega} n_\varepsilon  ^{{{{\alpha}-1}}}|\nabla n_\varepsilon|^2}
\\
\leq&\disp{\frac{{\alpha}}{2}\int_{\Omega} n_\varepsilon  ^{{{{\alpha}-1}}}|\nabla n_\varepsilon|^2+\frac{\alpha C_S^2}{2}\int_\Omega   n_\varepsilon^{1+\alpha}(1+n_\varepsilon)^{-2\alpha}
|\nabla c_\varepsilon|^2~~\mbox{for all}~~ t\in(0,T_{max,\varepsilon}).}\\
\end{array}
\label{3333cz2.5kkssss1214114114}
\end{equation}

Now, we must estimate the last term on the right-hand side of \dref{3333cz2.5kkssss1214114114}. To this end,
without loss of generality, we may assume $\alpha< 1$, since $\alpha \geq1$, can be
proved similarly and easily. In fact, if $\alpha\geq1$, then by $n_\varepsilon\geq0$,
$$
\begin{array}{rl}
&\disp{\frac{\alpha C_S^2}{2}\int_\Omega   n_\varepsilon^{1+\alpha}(1+n_\varepsilon)^{-2\alpha}
|\nabla c_\varepsilon|^2}\\
\leq&\disp{\frac{\alpha C_S^2}{2}\int_\Omega   n_\varepsilon(1+n_\varepsilon)^{-\alpha}
|\nabla c_\varepsilon|^2}\\
\leq&\disp{\frac{\alpha C_S^2}{2}\int_\Omega
|\nabla c_\varepsilon|^2  ,}
\end{array}
$$
so that, inserting the above inequality into \dref{3333cz2.5kkssss1214114114} and using \dref{bnmbncz2.5ghhjuyuivvbnnihjj}, we may derive
\dref{111334dddd444zjscz2.5297x9630222ssdd2114} and \dref{3.10gghhjuuloollsdffffffgghhhy} by using Lemma \ref{lemma630}.  \dref{334dddd444zjscz2.5297x9630222ssdd2114} can be proved by employing almost exactly the same arguments as in the proof of the case $0<\alpha<1$ (see \dref{ssdd3333cz2.5kkett677ddff734567789999001214114114rrggjjkk}--\dref{czfvgb2.5ghhddffggjuygdddhjjjuffghhhddfghhccvjkkklllhhjkkviihjj}). Therefore, we omit it.

While if $0<\alpha<1,$
 for any $\varepsilon_1>0,$
we invoke the Young inequality 
to find 
that
%
\begin{equation}
\begin{array}{rl}
&\disp{\frac{\alpha C_S^2}{2}\int_\Omega   n_\varepsilon^{1+\alpha}(1+n_\varepsilon)^{-2\alpha}
|\nabla c_\varepsilon|^2}\\
\leq&\disp{\frac{\alpha C_S^2}{2}\int_\Omega   n_\varepsilon(1+n_\varepsilon)^{-\alpha}
|\nabla c_\varepsilon|^2}\\
\leq&\disp{\frac{\alpha C_S^2}{2}\int_\Omega   n_\varepsilon^{1-\alpha}
|\nabla c_\varepsilon|^2}\\
 \leq&\disp{\varepsilon_1\int_\Omega   n_\varepsilon  ^{2{+\alpha}}+C_1(\varepsilon_1)\int_\Omega  |\nabla c_\varepsilon|^{\frac{2\alpha+4}{1+2\alpha}}  ,}
\end{array}
\label{3333cz2.563ss011228ddff}
\end{equation}
where  
$$C_1(\varepsilon_1)=\frac{1+2\alpha}{2{+\alpha}}\left(\varepsilon_1\frac{2+\alpha}{1-\alpha}\right)^{-\frac{1-\alpha}{1+2\alpha} }
\left(\frac{\alpha C_S^2}{2}\right)^{\frac{2{+\alpha}}{1+2\alpha} }.$$
%
%
In light of \dref{3.10gghhjuuloollgghhhyhh}, there exist positive constants $l_0>\frac{2-2\alpha}{3\alpha}$ and $C_2$ such  that
\begin{equation}\|c_\varepsilon(\cdot, t)\|_{L^{l_0}(\Omega)}\leq C_2~~ \mbox{for all}~~ t\in(0,T_{max,\varepsilon}).
\label{3.10gghhjukklllkklllokkllffghhjjoppuloollgghhhyhh}
\end{equation}

Next, with the help of the Gagliardo--Nirenberg inequality and \dref{3.10gghhjukklllkklllokkllffghhjjoppuloollgghhhyhh}, we derive that
\begin{equation}\label{ssdd3333cz2.5kkett677734567789999001214114114rrggjjkk}
\begin{array}{rl}
&\disp{C_1(\varepsilon_1)\int_\Omega  |\nabla c_\varepsilon|^{\frac{2\alpha+4}{1+2\alpha}}}
\\
\leq&\disp{C_3\|\Delta c_\varepsilon\|_{L^{2}(\Omega)}^{a\frac{2\alpha+4}{1+2\alpha}}\| c_\varepsilon\|_{L^{l_0}(\Omega)}^{(1-a)\frac{2\alpha+4}{1+2\alpha}}+C_3\| c_\varepsilon\|_{L^{l_0}(\Omega)}^{\frac{2\alpha+4}{1+2\alpha}}}\\
\leq&\disp{C_{4}\|\Delta c_\varepsilon\|_{L^{2}(\Omega)}^{a\frac{2\alpha+4}{1+2\alpha}}+C_{4}}\\
\end{array}
\end{equation}
with some positive constants $C_3$ and $C_{4}$, where
$$a=\frac{\frac{1}{2}+\frac{1}{l_0}-\frac{1+2\alpha}{2\alpha+4}}{\frac{1}{2}+\frac{1}{l_0}}\in(0,1).$$
We derive from the Young inequality that
\begin{equation}\label{ssdd3333cz2.5kkett677ddff734567789999001214114114rrggjjkk}
\begin{array}{rl}
&\disp{C_1(\varepsilon_1)\int_\Omega  |\nabla c_\varepsilon|^{\frac{2\alpha+4}{1+2\alpha}}\leq\frac{1}{4}\|\Delta c_\varepsilon\|_{L^{2}(\Omega)}^{2}+C_{5}}\\
\end{array}
\end{equation}
by using the fact that $a\frac{2\alpha+4}{1+2\alpha}<2$ due to $l_0>\frac{2-2\alpha}{3\alpha}$.
To estimate   $\Delta c_\varepsilon$,
taking $-\Delta{c_{\varepsilon}}$
as the test function for the second  equation of \dref{1.1fghyuisda},  using the Young   inequality yields  that for all $t\in(0,T_{max,\varepsilon})$
\begin{equation}
\begin{array}{rl}
\disp\frac{1}{{2}}\disp\frac{d}{dt}\|\nabla{c_{\varepsilon}}\|^{{{2}}}_{L^{{2}}(\Omega)}+
\int_{\Omega} |\Delta c_{\varepsilon}|^2+ \int_{\Omega} | \nabla c_{\varepsilon}|^2=&\disp{-\int_{\Omega} n_{\varepsilon}\Delta c_{\varepsilon}+\int_{\Omega} (u_{\varepsilon}\cdot\nabla c_{\varepsilon})\Delta c_{\varepsilon}}\\
=&\disp{-\int_{\Omega} n_{\varepsilon}\Delta c_{\varepsilon}-\int_{\Omega}\nabla c_{\varepsilon}\nabla (u_{\varepsilon}\cdot\nabla c_{\varepsilon})}\\
=&\disp{-\int_{\Omega} n_{\varepsilon}\Delta c_{\varepsilon}-\int_{\Omega}\nabla c_{\varepsilon}\nabla (\nabla u_{\varepsilon}\cdot\nabla c_{\varepsilon}),}\\
\end{array}
\label{hhxxcsssdfvvjjczddfdddfff2.5}
\end{equation}
because
$$
\begin{array}{rl}
&\disp{\int_{\Omega}\nabla c_{\varepsilon}\cdot(D^2 c_{\varepsilon}\cdot u_{\varepsilon})=\frac{1}{2}\int_{\Omega}  u_{\varepsilon}\cdot\nabla|\nabla c_{\varepsilon}|^2=0
~~\mbox{for all}~~ t\in(0,T_{max,\varepsilon})}\\
\end{array}
$$
due to the fact that $\nabla\cdot u_\varepsilon=0$.
Here since combining the Gagliardo-Nirenberg inequality with well known
elliptic regularity theory (\cite{Gilbarg4441215}) we can pick $C_{6}> 0$ such that
$$
\begin{array}{rl}
\disp \|\nabla c_{\varepsilon}\|_{L^{4}(\Omega)}^2\leq&\disp{C_{6}\|\Delta c_{\varepsilon}\|_{L^{2}(\Omega)}\|\nabla c_{\varepsilon}\|_{L^{2}(\Omega)}~~\mbox{for all}~~ t\in(0,T_{max,\varepsilon}),}\\
\end{array}
$$
which together with  the Cauchy-Schwarz inequality and the  Young  inequality implies that
\begin{equation}
\begin{array}{rl}
\disp-\int_{\Omega}\nabla c_{\varepsilon}\nabla (\nabla u_{\varepsilon}\cdot\nabla c_{\varepsilon})\leq&\disp{\|\nabla u_{\varepsilon}\|_{L^{2}(\Omega)}\|\nabla c_{\varepsilon}\|_{L^{4}(\Omega)}^2}\\
\leq&\disp{C_{6}\|\nabla u_{\varepsilon}\|_{L^{2}(\Omega)}\|\Delta c_{\varepsilon}\|_{L^{2}(\Omega)}\|\nabla c_{\varepsilon}\|_{L^{2}(\Omega)}}\\
\leq&\disp{C_{6}^2\|\nabla u_{\varepsilon}\|_{L^{2}(\Omega)}^2\|\nabla c_{\varepsilon}\|_{L^{2}(\Omega)}^2+\frac{1}{4}\|\Delta c_{\varepsilon}\|_{L^{2}(\Omega)}^2~~\mbox{for all}~~ t\in(0,T_{max,\varepsilon}).}\\
\end{array}
\label{hhxxcsssdfvvjjcddffzddfdddfff2.5}
\end{equation}
As the Cauchy-Schwarz inequality furthermore warrants that
\begin{equation}
\begin{array}{rl}
\disp-\int_{\Omega} n_{\varepsilon}\Delta c_{\varepsilon}\leq&\disp{\frac{1}{4}
\int_{\Omega}|\Delta c_{\varepsilon}|^2+\int_{\Omega}n_{\varepsilon}^2~~\mbox{for all}~~ t\in(0,T_{max,\varepsilon}),}\\
\end{array}
\label{hhxxcsssdfvvjjddddczddfdddfff2.5}
\end{equation}
from \dref{hhxxcsssdfvvjjczddfdddfff2.5} and \dref{hhxxcsssdfvvjjcddffzddfdddfff2.5} we thus infer that
\begin{equation}
\begin{array}{rl}
\disp\disp\frac{d}{dt}\|\nabla{c_{\varepsilon}}\|^{{{2}}}_{L^{{2}}(\Omega)}+
\int_{\Omega} |\Delta c_{\varepsilon}|^2+ 2\int_{\Omega} | \nabla c_{\varepsilon}|^2\leq&\disp{2\int_{\Omega}n_{\varepsilon}^2+
2C_{6}^2\|\nabla u_{\varepsilon}\|_{L^{2}(\Omega)}^2\|\nabla c_{\varepsilon}\|_{L^{2}(\Omega)}^2.}\\
\end{array}
\label{hhxxcsssdfvvsssjjczddfdddfff2.5}
\end{equation}
Collecting \dref{3333cz2.5kkssss1214114114}, \dref{ssdd3333cz2.5kkett677ddff734567789999001214114114rrggjjkk}--\dref{hhxxcsssdfvvsssjjczddfdddfff2.5}, we derive that for all $t\in(0,T_{max,\varepsilon})$,
\begin{equation}
\begin{array}{rl}
&\disp{\frac{d}{dt}(\|n_\varepsilon  \|^{{{1+\alpha}}}_{L^{{1+\alpha}}(\Omega)}+
\disp\|\nabla{c_{\varepsilon}}\|^{{{2}}}_{L^{{2}}(\Omega)})+
{\alpha}(1+\alpha)\int_{\Omega} n_\varepsilon  ^{{{{\alpha}-1}}}|\nabla n_\varepsilon|^2}
\\
&+\disp{\frac{1}{{2}}
\int_{\Omega} |\Delta c_{\varepsilon}|^2+2 \int_{\Omega} | \nabla c_{\varepsilon}|^2}\\
\leq&\disp{\varepsilon_1(1+\alpha)\int_\Omega   n_\varepsilon  ^{2{+\alpha}}+2\int_{\Omega}n_{\varepsilon}^2+
2C_{6}^2\|\nabla u_{\varepsilon}\|_{L^{2}(\Omega)}^2\|\nabla c_{\varepsilon}\|_{L^{2}(\Omega)}^2+C_{7},}\\
\end{array}
\label{3333cz2.5kkssss121hhjjj4114114}
\end{equation}
which combined with the Young inequality yields to
\begin{equation}
\begin{array}{rl}
&\disp{\frac{d}{dt}(\|n_\varepsilon  \|^{{{1+\alpha}}}_{L^{{1+\alpha}}(\Omega)}+
\disp\|\nabla{c_{\varepsilon}}\|^{{{2}}}_{L^{{2}}(\Omega)})+
{\alpha}(1+\alpha)\int_{\Omega} n_\varepsilon  ^{{{{\alpha}-1}}}|\nabla n_\varepsilon|^2}
\\
&+\disp{\frac{1}{{2}}
\int_{\Omega} |\Delta c_{\varepsilon}|^2+2 \int_{\Omega} | \nabla c_{\varepsilon}|^2}\\
\leq&\disp{2\varepsilon_1(1+\alpha)\int_\Omega   n_\varepsilon  ^{2{+\alpha}}+
2C_{6}^2\|\nabla u_{\varepsilon}\|_{L^{2}(\Omega)}^2\|\nabla c_{\varepsilon}\|_{L^{2}(\Omega)}^2+C_{8}~~\mbox{for all}~~ t\in(0,T_{max,\varepsilon})}\\
\end{array}
\label{3333cz2.5kksssssss121hhjjj4114114}
\end{equation}
by $\alpha>0$.

Now,  it follows from the Gagliardo--Nirenberg inequality, Lemma  \ref{2344ddfgczhhhh2.5ghju48cfg924ghyuji} that there exist  constants $\gamma_{0}> 0$ and $\gamma_{1} > 0$ such that
\begin{equation}\label{3333cz2.5kke345677ddff89001214114114rrggjjkk}
\begin{array}{rl}\disp\int_\Omega   n_\varepsilon  ^{2{+\alpha}}=& \| n_\varepsilon  ^{\frac{1+\alpha}{2}}\|_{L^{\frac{2\alpha+4}{1+\alpha}}(\Omega)}^{\frac{2\alpha+4}{1+\alpha}}\\
\leq& \gamma_{0}(\| \nabla n_\varepsilon  ^{\frac{1+\alpha}{2}}\|_{L^{2}(\Omega)}^{\frac{1+\alpha}{2+\alpha}} \| n_\varepsilon  ^{\frac{1+\alpha}{2}}\|_{L^{\frac{2}{1+\alpha}}(\Omega)}^{\frac{1}{2+\alpha}}+\| n_\varepsilon  ^{\frac{1+\alpha}{2}}\|_{L^{\frac{2}{1+\alpha}}(\Omega)})^{\frac{2\alpha+4}{1+\alpha}}\\
\leq& \gamma_{1}\| \nabla n_\varepsilon  ^{\frac{1+\alpha}{2}}\|_{L^{2}(\Omega)}^{2}+\gamma_{1}.\\
\end{array}
\end{equation}
We then achieve, with the help of \dref{3333cz2.5kke345677ddff89001214114114rrggjjkk}, that
which implies that
\begin{equation}\label{3333cz2.5kke345ddfff677ddff89001214114114rrggjjkk}
\begin{array}{rl}(1+\alpha){\alpha}\disp\int_{\Omega} n_\varepsilon  ^{{{{\alpha}-1}}}|\nabla n_\varepsilon|^2=&\disp\frac{4{\alpha}}{1+\alpha} \|\nabla n_\varepsilon  ^{\frac{1+\alpha}{2}}\|_{L^{2}(\Omega)}^{2}\\
\geq& \frac{1}{\gamma_{1}}\frac{4{\alpha}}{1+\alpha}(\disp\int_\Omega   n_\varepsilon  ^{2{+\alpha}}-1).\\
\end{array}
\end{equation}
By substituting \dref{3333cz2.5kke345ddfff677ddff89001214114114rrggjjkk} into \dref{3333cz2.5kksssssss121hhjjj4114114}, we find that
\begin{equation}
\begin{array}{rl}
&\disp{\frac{d}{dt}(\|n_\varepsilon  \|^{{{1+\alpha}}}_{L^{{1+\alpha}}(\Omega)}+
\|\nabla{c_{\varepsilon}}\|^{{{2}}}_{L^{{2}}(\Omega)})+
(\frac{1}{\gamma_{1}}\frac{4{\alpha}}{1+\alpha}-2(1+\alpha)\varepsilon_1)\int_\Omega   n_\varepsilon  ^{2{+\alpha}}}
\\
&+\disp{\frac{1}{{2}}
\int_{\Omega} |\Delta c_{\varepsilon}|^2+ 2\int_{\Omega} | \nabla c_{\varepsilon}|^2}\\
\leq&\disp{
2C_{6}^2\|\nabla u_{\varepsilon}\|_{L^{2}(\Omega)}^2\|\nabla c_{\varepsilon}\|_{L^{2}(\Omega)}^2+C_{9}~~\mbox{for all}~~ t\in(0,T_{max,\varepsilon}).}\\
\end{array}
\label{3333cz2.5kksssssss121hhjjj4sddfff114114}
\end{equation}
Choosing $\varepsilon_1=\frac{1}{\gamma_{1}}\frac{{\alpha}}{(1+\alpha)^2}$ in \dref{3333cz2.5kksssssss121hhjjj4sddfff114114}, we conclude that
\begin{equation}
\begin{array}{rl}
&\disp{\frac{d}{dt}(\|n_\varepsilon  \|^{{{1+\alpha}}}_{L^{{1+\alpha}}(\Omega)}+
\|\nabla{c_{\varepsilon}}\|^{{{2}}}_{L^{{2}}(\Omega)})+
\frac{1}{\gamma_{1}}\frac{2{\alpha}}{1+\alpha}\int_\Omega   n_\varepsilon  ^{2{+\alpha}}+ 2\int_{\Omega} | \nabla c_{\varepsilon}|^2}
\\
\leq&\disp{
2C_{6}^2\|\nabla u_{\varepsilon}\|_{L^{2}(\Omega)}^2\|\nabla c_{\varepsilon}\|_{L^{2}(\Omega)}^2+C_{9}~~\mbox{for all}~~ t\in(0,T_{max,\varepsilon}),}\\
\end{array}
\label{333ssss3cz2.5kksssssss121hhjjj4sddfff114114}
\end{equation}
which combined  the Young inequality implies that
\begin{equation}
\begin{array}{rl}
&\disp{\frac{d}{dt}(\|n_\varepsilon  \|^{{{1+\alpha}}}_{L^{{1+\alpha}}(\Omega)}+
\|\nabla{c_{\varepsilon}}\|^{{{2}}}_{L^{{2}}(\Omega)})+
2\int_\Omega   n_\varepsilon  ^{{{1+\alpha}}}+ 2\int_{\Omega} | \nabla c_{\varepsilon}|^2+\frac{1}{\gamma_{1}}\frac{{\alpha}}{1+\alpha}\int_\Omega   n_\varepsilon  ^{2{+\alpha}}}\\
\leq&\disp{
2C_{6}^2\|\nabla u_{\varepsilon}\|_{L^{2}(\Omega)}^2\|\nabla c_{\varepsilon}\|_{L^{2}(\Omega)}^2+C_{10}}\\
 \leq&\disp{
2C_{6}^2\|\nabla u_{\varepsilon}\|_{L^{2}(\Omega)}^2(\|\nabla c_{\varepsilon}\|_{L^{2}(\Omega)}^2+
 \|n_\varepsilon  \|^{{{1+\alpha}}}_{L^{{1+\alpha}}(\Omega)})+C_{10}~~\mbox{for all}~~ t\in(0,T_{max,\varepsilon}),}\\
\end{array}
\label{3333cz2.5kksssssss121hhjjj4sddfffdfff114114}
\end{equation}
where we have used the fact that $2\int_\Omega   n_\varepsilon  ^{{{1+\alpha}}}\leq \frac{1}{\gamma_{1}}\frac{{\alpha}}{1+\alpha}\int_\Omega   n_\varepsilon  ^{2{+\alpha}}+C_{10}$ by using $\alpha>0$ and the Young inequality.
Now,  again,  from the Gagliardo--Nirenberg inequality and Lemma  \ref{2344ddfgczhhhh2.5ghju48cfg924ghyuji} that there exist  constants $\gamma_{3}> 0$ and $\gamma_{4} > 0$ such that
\begin{equation}\label{3333cz2.5kke345677ddff89001ddff214114114rrggjjkk}
\begin{array}{rl}&\disp\int_{t}^{t+\tau}\int_\Omega   n_\varepsilon  ^{{1+\alpha}}\\
=& \int_{t}^{t+\tau}\|  n_\varepsilon  ^{\alpha}\|_{L^{\frac{1+\alpha}{\alpha}}(\Omega)}^{\frac{1+\alpha}{\alpha}}\\
\leq& \gamma_{3}(\int_{t}^{t+\tau}\| \nabla n_\varepsilon  ^{\alpha}\|_{L^{2}(\Omega)}^{\frac{{\alpha}}{1+\alpha}} \| n_\varepsilon  ^{\alpha}\|_{L^{\frac{1}{\alpha}}(\Omega)}^{\frac{1}{1+\alpha}}+\int_{t}^{t+\tau}\| n_\varepsilon  ^{\alpha}\|_{L^{\frac{1}{\alpha}}(\Omega)})^{\frac{2(1+\alpha)}{\alpha}}\\
\leq& \gamma_{4}\int_{t}^{t+\tau}\| \nabla n_\varepsilon  ^{\alpha}\|_{L^{2}(\Omega)}^{2}+\gamma_{4}~~\mbox{for all}~~ t\in(0,T_{max,\varepsilon}-\tau)\\
\end{array}
\end{equation}
by \dref{bnmbncz2.5ghhjuyuivvbnnihjj}, where $\tau=\min\{1,\frac{1}{6}T_{max,\varepsilon}\}.$
Therefore, by \dref{3333cz2.5kke345677ddff89001ddff214114114rrggjjkk}, we conclude that
\begin{equation}\label{3333cz2.kkk5kke345677ddfdddddf89001ddff214114114rrggjjkk}
\begin{array}{rl}\disp\int_{t}^{t+\tau}\int_\Omega   n_\varepsilon  ^{{{1+\alpha}}}\leq& \gamma_{5}~~\mbox{for all}~~ t\in(0,T_{max,\varepsilon}-\tau).\\
\end{array}
\end{equation}
Thus, if we write $y(t) :=\|n_\varepsilon(\cdot, t)\|^{{{1+\alpha}}}_{L^{{1+\alpha}}(\Omega)}+
\|\nabla{c_{\varepsilon}}(\cdot, t)\|^{{{2}}}_{L^{{2}}(\Omega)}$
and $\rho(t) =2C_{6}^2\int_{\Omega}|\nabla u_{\varepsilon}(\cdot, t)|^2$
for $t\in(0,T_{max,\varepsilon})$, so that,  \dref{3333cz2.5kksssssss121hhjjj4sddfffdfff114114}  implies that
\begin{equation}
\begin{array}{rl}
y'(t)+h(t)
\leq&\disp{ \rho(t)y(t)+C_{11}~\mbox{for all}~t\in(0,T_{max,\varepsilon}),}\\
\end{array}
\label{ddfghgfhggddhjjjnkkll11cz2.5ghju48}
\end{equation}
where $h(t)=
\frac{1}{\gamma_{1}}\frac{{\alpha}}{1+\alpha}\int_\Omega  n_\varepsilon^{2{+\alpha}}(\cdot,t)\geq0.$
Now, \dref{3333cz2.kkk5kke345677ddfdddddf89001ddff214114114rrggjjkk} as well as  \dref{bnmbncz2.5ghhjuyuivvbnnihjj}  ensure that for all $t\in(0,T_{max,\varepsilon}-\tau)$
\begin{equation}
\begin{array}{rl}
\int_{t}^{t+\tau}\rho(s)ds
\leq&\disp{ C_{12}}\\
\end{array}
\label{ddfghgffgghsdfdfffffgggggddhjjjhjjnnhhkklld911cz2.5ghju48}
\end{equation}
and
\begin{equation}
\begin{array}{rl}
\int_{t}^{t+\tau}y(s)ds
\leq&\disp{ C_{13}.}\\
\end{array}
\label{ddfghgffgghsdfffhhhhdfffffgggggddhjjjhjjnnhhkklld911cz2.5ghju48}
\end{equation}
For given $t\in (0, T_{max,\varepsilon})$,  applying \dref{3333cz2.kkk5kke345677ddfdddddf89001ddff214114114rrggjjkk} as well as  \dref{bnmbncz2.5ghhjuyuivvbnnihjj} again,  we can choose $t_0 \geq 0$ such that $t_0\in [t-\tau, t)$ and
\begin{equation}
\begin{array}{rl}
&\disp{ y(\cdot,t_0)\leq C_{14},}\\
\end{array}
\label{czfvgb2.5ghhjuyghjjjuffghhhddfghhccvjkkklllhhjkkviihjj}
\end{equation}
which combined with \dref{ddfghgfhggddhjjjnkkll11cz2.5ghju48} implies that
\begin{equation}
\begin{array}{rl}
y(t)\leq&\disp{y(t_0)e^{\int_{t_0}^t\rho(s)ds}+\int_{t_0}^te^{\int_{s}^t\rho(\tau)d\tau}C_{11}ds}\\
\leq&\disp{C_{14}e^{C_{12}}+\int_{t_0}^te^{C_{12}}C_{11}ds}\\
\leq&\disp{C_{14}e^{C_{12}}+e^{C_{12}}C_{11}~~\mbox{for all}~~t\in(0,T_{max,\varepsilon})}\\
\end{array}
\label{czfvgb2.5ghhddffggjuygdddhjjjuffghhhddfghhccvjkkklllhhjkkviihjj}
\end{equation}
by using the Gronwall lemma.
Hence \dref{111334dddd444zjscz2.5297x9630222ssdd2114} and \dref{334dddd444zjscz2.5297x9630222ssdd2114} are a direct consequence of \dref{czfvgb2.5ghhddffggjuygdddhjjjuffghhhddfghhccvjkkklllhhjkkviihjj}, and \dref{3.10gghhjuuloollsdffffffgghhhy} holds by
 integrating \dref{ddfghgfhggddhjjjnkkll11cz2.5ghju48} over $(t,t+\tau)$.
And
 the proof of this Lemma is completed.
\end{proof}

\section{Global Solvability of the Regularized Problem \dref{1.1fghyuisda}}

With the higher regularity for $n_{\varepsilon}$ obtained in Section 3 (see Lemma \ref{gddffffnjjmmx1.73ddddd1426677gg}),  we can derive the following Lemma
by using the Gagliardo--Nirenberg inequality and an application of well-known arguments
from parabolic regularity theory. Here {\bf the convexity of
$\Omega$} is not needed.

%

\begin{lemma}\label{lemma45630hhuujj}
Let $\alpha>0$ and $\gamma\in(\frac{1}{2},1).$  Then one can find a positive constant $C$ independent of $\varepsilon$
 such that 
\begin{equation}
\|n_\varepsilon(\cdot,t)\|_{L^\infty(\Omega)}  \leq C ~~\mbox{for all}~~ t\in(0,T_{max,\varepsilon}),
\label{zjscz2.5297x9630111kk}
\end{equation}
\begin{equation}
\|c_\varepsilon(\cdot,t)\|_{W^{1,\infty}(\Omega)}  \leq C ~~\mbox{for all}~~ t\in(0,T_{max,\varepsilon})
\label{zjscz2.5297x9630111kkhh}
\end{equation}
as well as
\begin{equation}
\|u_\varepsilon(\cdot,t)\|_{L^{\infty}(\Omega)}  \leq C ~~\mbox{for all}~~ t\in(0,T_{max,\varepsilon})
\label{zjscz2.5297x9630111kkhhffrr}
\end{equation}
and
\begin{equation}
\|A^\gamma u_\varepsilon(\cdot,t)\|_{L^{2}(\Omega)}  \leq C ~~\mbox{for all}~~ t\in(0,T_{max,\varepsilon}).
\label{zjscz2.5297x9630111kkhhffrreerr}
\end{equation}
Moreover,  for all $p > 1$, there exists $C(p) > 0$ satisfying
\begin{equation}
\|\nabla u_\varepsilon(\cdot,t)\|_{L^{p}(\Omega)}  \leq C(p) ~~\mbox{for all}~~ t\in(0,T_{max,\varepsilon}).
\label{zjscz2.5297x963011dkllldfff1kkhh}
\end{equation}
\end{lemma}
\begin{proof}
In the  following, we let $C_i(i\in \mathbb{N})$ denote some different constants, which are independent of $\varepsilon$, and if no special explanation, they depend at most on $\Omega, \phi, \alpha, n_0, c_0$ and
$u_0$.

{\bf Step 1. The boundedness of $\|\nabla u_{\varepsilon}(\cdot, t)\|_{L^2(\Omega)}$  for all $t\in (0, T_{max,\varepsilon})$}

 Firstly, in view of \dref{bnmbncz2.5ghhjuyuivvbnnihjj} and \dref{3.10gghhjuuloollsdffffffgghhhy}, we derive there exists a positive constant $\alpha_0$ such that
\begin{equation}\label{3333cz2.kkk5kke345677ddfddffddddf89001ddff214114114rrggjjkk}
\begin{array}{rl}\disp\int_{t}^{t+\tau}\int_\Omega  |\nabla {u_{\varepsilon}}|^2\leq& \alpha_{0}~~\mbox{for all}~~ t\in(0,T_{max,\varepsilon}-\tau)\\
\end{array}
\end{equation}
and
\begin{equation}\label{3333cz2.kkk5kke345fffff677ddfdddddf89001ddff214114114rrggjjkk}
\begin{array}{rl}\disp\int_{t}^{t+\tau}\int_\Omega   n_\varepsilon  ^{2}\leq& \alpha_{0}~~\mbox{for all}~~ t\in(0,T_{max,\varepsilon}-\tau)\\
\end{array}
\end{equation}
with $\tau=\min\{1,\frac{1}{6}T_{max,\varepsilon}\}.$
Here we have used the H\"{o}lder inequality and $\alpha>0$.

Testing the projected Stokes equation $u_{\varepsilon t} +Au_{\varepsilon} =
 \mathcal{P}[-\kappa (Y_{\varepsilon}u_{\varepsilon} \cdot \nabla)u_{\varepsilon}+n_{\varepsilon}\nabla \phi]$
by $Au_{\varepsilon}$,  integrating by parts, we derive
%
%
%
%
\begin{equation}
\begin{array}{rl}
&\disp{\frac{1}{{2}}\frac{d}{dt}\|A^{\frac{1}{2}}u_{\varepsilon}\|^{{{2}}}_{L^{{2}}(\Omega)}+
\int_{\Omega}|Au_{\varepsilon}|^2 }\\
=&\disp{ \int_{\Omega}Au_{\varepsilon}\mathcal{P}(-\kappa
(Y_{\varepsilon}u_{\varepsilon} \cdot \nabla)u_{\varepsilon})+ \int_{\Omega}\mathcal{P}(n_{\varepsilon}\nabla\phi) Au_{\varepsilon}}\\
\leq&\disp{ \frac{1}{4}\int_{\Omega}|Au_{\varepsilon}|^2+2\kappa^2\int_{\Omega}
|(Y_{\varepsilon}u_{\varepsilon} \cdot \nabla)u_{\varepsilon}|^2+2 \|\nabla\phi\|^2_{L^\infty(\Omega)}\int_{\Omega}n_{\varepsilon}^2~
\mbox{for all}~t\in(0,T_{max,\varepsilon})}\\
\end{array}
\label{ddfghgffgghgghjjnnhhkklld911cz2.5ghju48}
\end{equation}
by the Young inequality.
However, using the Cauchy-Schwarz inequality,
there exists a positive constant $C_1$ and $C_{2}$
such that
\begin{equation}
\begin{array}{rl}
&2\kappa^2\disp\int_{\Omega}
|(Y_{\varepsilon}u_{\varepsilon} \cdot \nabla)u_{\varepsilon}|^2\\
\leq&\disp{ 2\kappa^2\|Y_{\varepsilon}u_{\varepsilon}\|^2_{L^4(\Omega)}\|\nabla u_{\varepsilon}\|^2_{L^4(\Omega)}}\\
\leq&\disp{ 2\kappa^2C_1[\|\nabla Y_{\varepsilon}u_{\varepsilon}\|_{L^2(\Omega)}\|Y_{\varepsilon}u_{\varepsilon}\|_{L^2(\Omega)}][\|A u_{\varepsilon}\|_{L^2(\Omega)}\|\nabla u_{\varepsilon}\|_{L^2(\Omega)}]
}\\
\leq&\disp{ 2\kappa^2C_1C_{2}\|\nabla Y_{\varepsilon}u_{\varepsilon}\|_{L^2(\Omega)}[\|A u_{\varepsilon}\|_{L^2(\Omega)}\|\nabla u_{\varepsilon}\|_{L^2(\Omega)}]
~~\mbox{for all}~~t\in(0,T_{max,\varepsilon})}\\
\end{array}
\label{ssdcfvgddfghgghjd911cz2.5ghju48}
\end{equation}
by  \dref{czfvgb2.5ghhjuyuccvviihjj} the fact that $\|Y_{\varepsilon}u_{\varepsilon}\|_{L^2(\Omega)}\leq
\|u_{\varepsilon}\|_{L^2(\Omega)}.$
Now, from  $D( A^{\frac{1}{2}})  :=W^{1,2}_0(\Omega;\mathbb{R}^2) \cap L_{\sigma}^{2}(\Omega)$ and
 \dref{czfvgb2.5ghhjuyuccvviihjj}, it follows that 
\begin{equation}
\|\nabla Y_{\varepsilon}u_{\varepsilon}\|_{L^2(\Omega)}=\|A^{\frac{1}{2}} Y_{\varepsilon}u_{\varepsilon}\|_{L^2(\Omega)}=\|Y_{\varepsilon} A^{\frac{1}{2}} u_{\varepsilon}\|_{L^2(\Omega)}\leq \|A^{\frac{1}{2}}  u_{\varepsilon}\|_{L^2(\Omega)}\leq \|\nabla  u_{\varepsilon}\|_{L^2(\Omega)}.
\label{ssdcfdhhgghjjnnhhkklld911cz2.5ghju48}
\end{equation}
Here by  Theorem 2.1.1 of \cite{Sohr}, we derive  that $\|A(\cdot)\|_{L^{2}(\Omega)}$ defines a norm
equivalent to $\|\cdot\|_{W^{2,2}(\Omega)}$ on $D(A)$.
Combining with this and
substituting \dref{ssdcfdhhgghjjnnhhkklld911cz2.5ghju48} into \dref{ssdcfvgddfghgghjd911cz2.5ghju48} yields
\begin{equation}
\begin{array}{rl}
2\kappa^2\disp\int_{\Omega}
|(Y_{\varepsilon}u_{\varepsilon} \cdot \nabla)u_{\varepsilon}|^2
\leq&\disp{ C_{3}\|A u_{\varepsilon}\|_{L^2(\Omega)}\|\nabla u_{\varepsilon}\|_{L^2(\Omega)}^2}\\
\leq&\disp{ \frac{1}{4}\|A u_{\varepsilon}\|_{L^2(\Omega)}+C_{3}^2\|\nabla u_{\varepsilon}\|_{L^2(\Omega)}^4
~~\mbox{for all}~~t\in(0,T_{max,\varepsilon})}\\
\end{array}
\label{ssdcfvgddfghggkllllllhjd911cz2.5ghju48}
\end{equation}
by using the Young inequality,
where $C_{3}= 2\kappa^2C_{1}C_{2}.$
Thus, if we write $z(t) :=\int_{\Omega}|\nabla u_{\varepsilon}(\cdot, t)|^2$
and $\rho(t) =2C_{3}^2\int_{\Omega}|\nabla u_{\varepsilon}(\cdot, t)|^2$
for $t\in(0,T_{max,\varepsilon})$, then \dref{ssdcfvgddfghgghjd911cz2.5ghju48} along with \dref{ddfghgffgghgghjjnnhhkklld911cz2.5ghju48} implies that
\begin{equation}
\begin{array}{rl}
z'(t)\leq\rho(t)z(t)+h(t)&\disp{~~\mbox{for all}~~t\in(0,T_{max,\varepsilon}),}\\
\end{array}
\label{ddfghgffgghggddhjjjhjjnnhhkklld911cz2.5ghju48}
\end{equation}
where $h(t)=4 \|\nabla\phi\|^2_{L^\infty(\Omega)}\int_{\Omega}n_{\varepsilon}^2(\cdot,t).$
Here we have used the fact that $\|A^{\frac{1}{2}}u_{\varepsilon}\|^{{{2}}}_{L^{{2}}(\Omega)} = \|\nabla u_{\varepsilon}\|^{{{2}}}_{L^{{2}}(\Omega)}.$
Now, \dref{3333cz2.kkk5kke345677ddfddffddddf89001ddff214114114rrggjjkk}
and \dref{3333cz2.kkk5kke345fffff677ddfdddddf89001ddff214114114rrggjjkk}    ensure that for all $t\in(0,T_{max,\varepsilon}-\tau)$
\begin{equation}
\begin{array}{rl}
\int_{t}^{t+\tau}\rho(s)ds
\leq&\disp{ 2C_{3}^2\alpha_0}\\
\end{array}
\label{ddfghgffgghsdfdfffffgggggddhjjjhjjnnhhkklld911cz2.5ghju48}
\end{equation}
and
\begin{equation}
\begin{array}{rl}
\int_{t}^{t+\tau}h(s)ds
\leq&\disp{ 4 \|\nabla\phi\|^2_{L^\infty(\Omega)}\alpha_0.}\\
\end{array}
\label{ddfghgffgghsdfffhhhhdfffffgggggddhjjjhjjnnhhkklld911cz2.5ghju48}
\end{equation}
For given $t\in (0, T_{max,\varepsilon})$,  applying \dref{3333cz2.kkk5kke345677ddfddffddddf89001ddff214114114rrggjjkk} again,  we can choose $t_0 \geq 0$ such that $t_0\in [t-\tau, t)$ and
\begin{equation}
\begin{array}{rl}
&\disp{\int_{\Omega}   |\nabla u_{\varepsilon}(\cdot,t_0)|^2\leq C_{4},}\\
\end{array}
\label{czfvgb2.5ghhjuyghjjjuffghhhddfghhccvjkkklllhhjkkviihjj}
\end{equation}
which combined with \dref{ddfghgffgghggddhjjjhjjnnhhkklld911cz2.5ghju48} implies that
\begin{equation}
\begin{array}{rl}
z(t)\leq&\disp{z(t_0)e^{\int_{t_0}^t\rho(s)ds}+\int_{t_0}^te^{\int_{s}^t\rho(\tau)d\tau}h(s)ds}\\
\leq&\disp{C_{4}e^{2C_{3}^2\alpha_0}+\int_{t_0}^te^{2C_{3}^2\alpha_0}h(s)ds}\\
\leq&\disp{C_{4}e^{2C_{3}^2\alpha_0}+e^{2C_{3}^2\alpha_0}4 \|\nabla\phi\|^2_{L^\infty(\Omega)}\alpha_{0}~~\mbox{for all}~~t\in(0,T_{max,\varepsilon})}\\
\end{array}
\label{czfvgb2.5ghhddffggjuyghjjjuffghhhddfghhccvjkkklllhhjkkviihjj}
\end{equation}
by integration.

{\bf Step 2. The boundedness of $\|\nabla c_{\varepsilon}(\cdot, t)\|_{L^{q_0}(\Omega)}$ 
   for all $t\in (0, T_{max,\varepsilon})$ and some $q_0>2$}

Next, since,  $W^{1,2}(\Omega)\hookrightarrow L^p(\Omega)$ for any $p>1,$ therefore, the  boundedness of $\|\nabla u_{\varepsilon}(\cdot, t)\|_{L^2(\Omega)}$ as well as  $\|\nabla c_{\varepsilon}(\cdot, t)\|_{L^2(\Omega)}$ (see \dref{334dddd444zjscz2.5297x9630222ssdd2114}) yields  that there exists a positive constant $C_{28}$ such that
\begin{equation}
\begin{array}{rl}
\|u_\varepsilon(\cdot, t)\|_{L^{p}(\Omega)}+\|c_\varepsilon(\cdot, t)\|_{L^{p}(\Omega)}\leq  C_{5}~~ \mbox{for all}~~ t\in(0,T_{max,\varepsilon})~~~\mbox{and}~~~p>1\\
\end{array}
\label{cz2.5jkkcvddffggggddfghhhjjhhdfffjjkvvhjkfffffkhhgll}
\end{equation}
by using the Poincar\'{e} inequality and \dref{czfvgb2.5ghhjuyuccvviihjj}.

Without loss of generality,  suppose  that $\alpha< 1$, since $\alpha \geq1$, can be
proved similarly and easily.
On the other hand,
from the
constants formula for $c_\varepsilon$, again, we derive the H\"{o}lder inequality that
\begin{equation}
\begin{array}{rl}
\disp{\|\nabla c_\varepsilon(\cdot, t)\|_{L^{\bar{q}_0}(\Omega)}}\leq&\disp{\|\nabla e^{t(\Delta-1)} c_0\|_{L^{\bar{q}_0}(\Omega)}+
\int_{0}^t\|\nabla e^{(t-s)(\Delta-1)}(n_\varepsilon(\cdot,s)\|_{L^{\bar{q}_0}(\Omega)}ds}\\
&\disp{+\int_{0}^t\|\nabla e^{(t-s)(\Delta-1)}\nabla \cdot(u_{\varepsilon}(\cdot,s) c_{\varepsilon}(\cdot,s))\|_{L^{\bar{q}_0}(\Omega)}ds,}\\
\end{array}
\label{11144444zjccfgghhhfgbhjcvvvbscz2.5297x96301ku}
\end{equation}
where $2<\bar{q}_0:=\frac{2}{1-\alpha}<\frac{2(1+\alpha)}{1-\alpha}$ by $0<\alpha<1$.
Hence, we  use Lemma \ref{llssdrffmmggnnccvvccvvkkkkgghhkkllvvlemma45630}  again   to obtain there
 exist positive constants $C_{i}(i=6\cdots 12)$ and $\lambda$ such that
\begin{equation}
\begin{array}{rl}
\|\nabla e^{t(\Delta-1)} c_0\|_{L^{\bar{q}_0}(\Omega)}\leq C_{6}~~ \mbox{for all}~~ t\in(0,T_{max,\varepsilon})\\
\end{array}
\label{zjccffgbhjcghhhjjjvvvbscz2.5297x96301ku}
\end{equation}
\begin{equation}
\begin{array}{rl}
&\disp{\int_{0}^t\|\nabla e^{(t-s)(\Delta-1)}n_\varepsilon(\cdot,s)\|_{L^{\bar{q}_0}(\Omega)}ds}\\
\leq&\disp{C_{7}\int_{0}^t[1+(t-s)^{-\frac{1}{2}-\frac{2}{2}(\frac{1}{1+\alpha}-\frac{1}{\bar{q}_0})}] e^{-\lambda(t-s)}
\|n_\varepsilon(\cdot,s)\|_{L^{1+\alpha}(\Omega)}ds}\\
\leq&\disp{C_{8}~~ \mbox{for all}~~ t\in(0,T_{max,\varepsilon})}\\
\end{array}
\label{zjccffgbhjcvvvbscz2.5297x96301ku}
\end{equation}
as well as
\begin{equation}
\begin{array}{rl}
&\disp{\int_{0}^t\|\nabla e^{(t-s)(\Delta-1)}\nabla \cdot(u_\varepsilon(\cdot,s) c_\varepsilon(\cdot,s))\|_{L^{\bar{q}_0}(\Omega)}ds}\\
\leq&\disp{C_{9}\int_{0}^t\|(-\Delta+1)^\iota e^{(t-s)(\Delta-1)}\nabla \cdot(u_\varepsilon(\cdot,s) c_\varepsilon(\cdot,s))\|_{L^{6}(\Omega)}ds}\\
\leq&\disp{C_{10}\int_{0}^t(t-s)^{-\iota-\frac{1}{2}-\tilde{\kappa}} e^{-\lambda(t-s)}\|u_\varepsilon(\cdot,s) c_\varepsilon(\cdot,s)\|_{L^{6}(\Omega)}ds}\\
\leq&\disp{C_{11}\int_{0}^t(t-s)^{-\iota-\frac{1}{2}-\tilde{\kappa}} e^{-\lambda(t-s)}\|u_\varepsilon(\cdot,s)\|_{L^{12}(\Omega)}\| c_\varepsilon(\cdot,s)\|_{L^{12}(\Omega)}ds}\\
\leq&\disp{C_{12}~~ \mbox{for all}~~ t\in(0,T_{max,\varepsilon})}\\
\end{array}
\label{1111zjccffgbhjcvdgghhhhdfgghhvvbscz2.5297x96301ku}
\end{equation}
by using \dref{cz2.5jkkcvddffggggddfghhhjjhhdfffjjkvvhjkfffffkhhgll}, \dref{ccvvx1.731426677gg} as well as  $\frac{1}{2}+\frac{2}{2}(\frac{1}{6}-\frac{1}{\bar{q}_0})<\iota$ and $\min\{-\iota-\frac{1}{2}-\tilde{\kappa},
-\frac{1}{2}-\frac{2}{2}(\frac{1}{1+\alpha}-\frac{1}{\bar{q}_0})\}>-1$,
where $\iota=\frac{11}{24},\tilde{\kappa}=\frac{1}{48}$.
Therefore,
collecting  \dref{11144444zjccfgghhhfgbhjcvvvbscz2.5297x96301ku}--\dref{1111zjccffgbhjcvdgghhhhdfgghhvvbscz2.5297x96301ku}, yields to 
\begin{equation}
\begin{array}{rl}
\|\nabla c_\varepsilon(\cdot, t)\|_{L^{\bar{q}_0}(\Omega)}\leq  C_{13}~~ \mbox{for all}~~ t\in(0,T_{max,\varepsilon}).\\
\end{array}
\label{cz2.5jkkcvddffgggghhdfffjjkvvhjkfffffkhhgll}
\end{equation}
{\bf Step 3. The boundedness of $\|n_{\varepsilon}(\cdot, t)\|_{L^{\infty}(\Omega)}$  for all  $t\in (0, T_{max,\varepsilon})$}

Fix $T\in (0, T_{max,\varepsilon})$.  Let $M(T):=\sup_{t\in(0,T)}\|n_{\varepsilon}(\cdot,t)\|_{L^\infty(\Omega)}$ and $\tilde{h}_{\varepsilon}:=S_\varepsilon(x, n_{\varepsilon}, c_{\varepsilon})\nabla c_{\varepsilon}+u_\varepsilon$. Then by  \dref{cz2.5jkkcvddffgggghhdfffjjkvvhjkfffffkhhgll}, \dref{x1.73142vghf48gg}  and \dref{cz2.5jkkcvddffggggddfghhhjjhhdfffjjkvvhjkfffffkhhgll},
there exists $C_{14} > 0$ such that
\begin{equation}
\begin{array}{rl}
\|\tilde{h}_{\varepsilon}(\cdot, t)\|_{L^{q_0}(\Omega)}\leq&\disp{C_{14}~~ \mbox{for all}~~ t\in(0,T_{max,\varepsilon}),}\\
\end{array}
\label{cz2ddff.57151ccvhhjjjkkkuuifghhhivhccvvhjjjkkhhggjjllll}
\end{equation}
where $q_0=\frac{2}{1-\alpha}$.
Hence, due to the fact that $\nabla\cdot u_{\varepsilon}=0$,  again,  by means of an
associate variation-of-constants formula for $c_{\varepsilon}$, we can derive
\begin{equation}
n_{\varepsilon}(t)=e^{(t-t_0)\Delta}n_{\varepsilon}(\cdot,t_0)-\int_{t_0}^{t}e^{(t-s)\Delta}\nabla\cdot(n_{\varepsilon}(\cdot,s)\tilde{h}_{\varepsilon}(\cdot,s)) ds,~~ t\in(t_0, T),
\label{5555fghbnmcz2.5ghjjjkkklu48cfg924ghyuji}
\end{equation}
where $t_0 := (t-1)_{+}$.
If $t\in(0,1]$,
by virtue of the maximum principle, we derive that
\begin{equation}
\begin{array}{rl}
\|e^{(t-t_0)\Delta}n_{\varepsilon}(\cdot,t_0)\|_{L^{\infty}(\Omega)}\leq &\disp{\|n_0\|_{L^{\infty}(\Omega)},}\\
\end{array}
\label{zjccffgbhjffghhjcghhhjjjvvvbscz2.5297x96301ku}
\end{equation}
while if $t > 1$ then in light of the  $L^p$-$L^q$ estimates for the Neumann heat semigroup and Lemma \ref{fvfgfflemma45}, we conclude that
\begin{equation}
\begin{array}{rl}
\|e^{(t-t_0)\Delta}n_{\varepsilon}(\cdot,t_0)\|_{L^{\infty}(\Omega)}\leq &\disp{C_{15}(t-t_0)^{-\frac{2}{2}}\|n_{\varepsilon}(\cdot,t_0)\|_{L^{1}(\Omega)}\leq C_{16}.}\\
\end{array}
\label{zjccffgbhjffghhjcghghjkjjhhjjjvvvbscz2.5297x96301ku}
\end{equation}
The last term on the right-hand side of \dref{5555fghbnmcz2.5ghjjjkkklu48cfg924ghyuji} is estimated as follows.
Fix an arbitrary $p\in(2,q_0)$ and then once more invoke known smoothing
properties of the
Stokes semigroup  and the H\"{o}lder inequality to find $C_{17} > 0$ and $C_{18} > 0$  such that
\begin{equation}
\begin{array}{rl}
&\disp\int_{t_0}^t\| e^{(t-s)\Delta}\nabla\cdot(n_{\varepsilon}(\cdot,s)\tilde{h}_{\varepsilon}(\cdot,s)\|_{L^\infty(\Omega)}ds\\
\leq&\disp C_{17}\int_{t_0}^t(t-s)^{-\frac{1}{2}-\frac{2}{2p}}\|n_{\varepsilon}(\cdot,s)\tilde{h}_{\varepsilon}(\cdot,s)\|_{L^p(\Omega)}ds\\
\leq&\disp C_{17}\int_{t_0}^t(t-s)^{-\frac{1}{2}-\frac{2}{2p}}\| n_{\varepsilon}(\cdot,s)\|_{L^{\frac{pq_0}{q_0-p}}(\Omega)}\|\tilde{h}_{\varepsilon}(\cdot,s)\|_{L^{q_0}(\Omega)}ds\\
\leq&\disp C_{17}\int_{t_0}^t(t-s)^{-\frac{1}{2}-\frac{2}{2p}}\| u_{\varepsilon}(\cdot,s)\|_{L^{\infty}(\Omega)}^b\| u_{\varepsilon}(\cdot,s)\||_{L^1(\Omega)}^{1-b}\|\tilde{h}_{\varepsilon}(\cdot,s)\|_{L^{q_0}(\Omega)}ds\\
\leq&\disp C_{18}M^b(T)~~\mbox{for all}~~ t\in(0, T),\\
\end{array}
\label{ccvbccvvbbnnndffghhjjvcvvbccfbbnfgbghjjccmmllffvvggcvvvvbbjjkkdffzjscz2.5297x9630xxy}
\end{equation}
where $b:=\frac{pq_0-q_0+p}{pq_0}\in(0,1)$.
Since $p>2$, we conclude that
$-\frac{1}{2}-\frac{2}{2p}>-1$.
In combination with \dref{5555fghbnmcz2.5ghjjjkkklu48cfg924ghyuji}--\dref{ccvbccvvbbnnndffghhjjvcvvbccfbbnfgbghjjccmmllffvvggcvvvvbbjjkkdffzjscz2.5297x9630xxy} and using the definition of $M(T)$
we obtain
$C_{19} > 0$ such that
\begin{equation}
\begin{array}{rl}
&\disp  M(T)\leq C_{19}+C_{19}M^b(T)~~\mbox{for all}~~ T\in(0, T_{max,\varepsilon}).\\
\end{array}
\label{ccvbccvvbbnnndffghhjjvcvvfghhhbccfbbnfgbghjjccmmllffvvggcvvvvbbjjkkdffzjscz2.5297x9630xxy}
\end{equation}
Hence, with the help of $b<1$,  in light of  $T\in (0, T_{max,\varepsilon})$ was arbitrary,
we can get
\begin{equation}
\begin{array}{rl}
\|n_{\varepsilon}(\cdot, t)\|_{L^{\infty}(\Omega)}\leq&\disp{C_{20}~~ \mbox{for all}~~ t\in(0,T_{max,\varepsilon})}\\
\end{array}
\label{cz2.57ghhhh151ccvhhjjjkkkffgghhuuiivhccvvhjjjkkhhggjjllll}
\end{equation}
by   some basic calculation.

{\bf Step 4. The boundedness of $\|A^\gamma u_{\varepsilon}(\cdot, t)\|_{L^2(\Omega)}$ (with $\gamma\in ( \frac{1}{2}, 1)$) and $\| u_{\varepsilon}(\cdot, t)\|_{L^{\infty}(\Omega)}$ for all $t\in (0, T_{max,\varepsilon})$}

Applying  the variation-of-constants formula for the projected version of the third
equation in \dref{1.1fghyuisda}, we derive that
$$u_\varepsilon(\cdot, t) = e^{-tA}u_0 +\int_0^te^{-(t-\tau)A}
\mathcal{P}[n_\varepsilon(\cdot,t)\nabla\phi-\kappa
(Y_{\varepsilon}u_{\varepsilon} \cdot \nabla)u_{\varepsilon}]d\tau~~ \mbox{for all}~~ t\in(0,T_{max,\varepsilon}).$$
Therefore, with the help of the  standard smoothing
properties of the Stokes semigroup we derive  that for all $t\in(0,T_{max,\varepsilon})$ and $\gamma\in ( \frac{1}{2}, 1)$,
there exist $C_{21} > 0$ and $C_{22} > 0$ such that
\begin{equation}
\begin{array}{rl}
\|A^\gamma u_{\varepsilon}(\cdot, t)\|_{L^2(\Omega)}\leq&\disp{\|A^\gamma
e^{-tA}u_0\|_{L^2(\Omega)} +\int_0^t\|A^\gamma e^{-(t-\tau)A}h_{\varepsilon}(\cdot,\tau)d\tau\|_{L^2(\Omega)}d\tau}\\
\leq&\disp{\|A^\gamma u_0\|_{L^2(\Omega)} +C_{21}\int_0^t(t-\tau)^{-\gamma-\frac{2}{2}(\frac{1}{p_0}-\frac{1}{2})}e^{-\lambda(t-\tau)}\|h_{\varepsilon}(\cdot,\tau)\|_{L^{p_0}(\Omega)}d\tau}\\
\leq&\disp{C_{22}+C_{21}\int_0^t(t-\tau)^{-\gamma-\frac{2}{2}(\frac{1}{p_0}-\frac{1}{2})}e^{-\lambda(t-\tau)}\|h_{\varepsilon}(\cdot,\tau)
\|_{L^{p_0}(\Omega)}d\tau}\\
\end{array}
\label{ssddcz2.571hhhhh51222ccvvhddfccvvhjjjkkhhggjjllll}
\end{equation}
by using \dref{ccvvx1.731426677gg},
where   $h_{\varepsilon}=\mathcal{P}[n_\varepsilon(\cdot,t)\nabla\phi-\kappa
(Y_{\varepsilon}u_{\varepsilon} \cdot \nabla)u_{\varepsilon}]$ and $p_0\in (1,2)$ which satisfies 
\begin{equation}p_0>\frac{2}{3-2\gamma}.
\label{cz2.571hhhhh51222ccvvhddfccffgghhhhvvhjjjkkhhggjjllll}
\end{equation}
Now, in order to estimate $\|h_{\varepsilon}(\cdot,\tau)
\|_{L^{p_0}(\Omega)}$, we use the H\"{o}lder inequality and the continuity of $\mathcal{P}$ in $L^p(\Omega;\mathbb{R}^2)$ (see \cite{Fujiwara66612186}) as well as the boundedness of $\|n_{\varepsilon}(\cdot, t)\|_{L^{\infty}(\Omega)}$ (for all  $t\in (0, T_{max,\varepsilon})$), we see that
there exist $C_{23},C_{24},C_{25}, C_{26}> 0$ and  $C_{27} > 0$ that
\begin{equation}
\begin{array}{rl}
\|h_{\varepsilon}(\cdot,t)\|_{L^{p_0}(\Omega)}\leq& C_{23}\|(Y_{\varepsilon}u_{\varepsilon} \cdot \nabla)u_{\varepsilon}(\cdot,t)\|_{L^{p_0}(\Omega)}+ C_{23}\|n_\varepsilon(\cdot,t)\|_{L^{p_0}(\Omega)}\\
\leq& C_{24}\|Y_{\varepsilon}u_{\varepsilon}\|_{L^{\frac{2p_0}{2-p_0}}(\Omega)} \|\nabla u_{\varepsilon}(\cdot,t)\|_{L^{2}(\Omega)}+ C_{24}\\
\leq& C_{25}\|\nabla Y_{\varepsilon}u_{\varepsilon}\|_{L^{2}(\Omega)} \|\nabla u_{\varepsilon}(\cdot,t)\|_{L^{2}(\Omega)}+ C_{24}\\
\leq& C_{26}\|\nabla u_{\varepsilon}\|_{L^{2}(\Omega)}^2+ C_{24}\\
\leq& C_{27}~~~\mbox{for all}~~ t\in(0,T_{max,\varepsilon})\\
\end{array}
\label{cz2.571hhhhh5122ddfddfffgg2ccvvhddfccffgghhhhvvhjjjkkhhggjjllll}
\end{equation}
by using $W^{1,2}(\Omega)\hookrightarrow L^\frac{2p_0}{2-p_0}(\Omega)$ and the boundedness of $\|\nabla u_{\varepsilon}(\cdot,t)\|_{L^{2}(\Omega)}.$
Inserting \dref{cz2.571hhhhh5122ddfddfffgg2ccvvhddfccffgghhhhvvhjjjkkhhggjjllll} into \dref{ssddcz2.571hhhhh51222ccvvhddfccvvhjjjkkhhggjjllll} and applying \dref{cz2.571hhhhh51222ccvvhddfccffgghhhhvvhjjjkkhhggjjllll}, we conclude that
\begin{equation}
\begin{array}{rl}
\|A^\gamma u_{\varepsilon}(\cdot, t)\|_{L^2(\Omega)}\leq&\disp{C_{28}\int_0^t(t-\tau)^{-\gamma-\frac{2}{2}(\frac{1}{p_0}-\frac{1}{2})}e^{-\lambda(t-\tau)}\|h_{\varepsilon}(\cdot,\tau)
\|_{L^{p_0}(\Omega)}d\tau,~~~\mbox{for all}~~ t\in(0,T_{max,\varepsilon}),}\\
\end{array}
\label{ssddcz2.571hhjjjhhhhh51222ccvvhddfccvvhjjjksdddffkhhggjjllll}
\end{equation}
where we have used the fact that
$$\begin{array}{rl}\disp\int_{0}^t(t-\tau)^{-\gamma-\frac{2}{2}(\frac{1}{p_0}-\frac{1}{2})}e^{-\lambda(t-\tau)}ds
\leq&\disp{\int_{0}^{\infty}\sigma^{-\gamma-\frac{2}{2}(\frac{1}{p_0}-\frac{1}{2})} e^{-\lambda\sigma}d\sigma<+\infty}\\
\end{array}
$$
by \dref{cz2.571hhhhh51222ccvvhddfccffgghhhhvvhjjjkkhhggjjllll}.
Observe that  $D(A^\gamma)$ is continuously embedded into $L^\infty(\Omega)$ by $\gamma>\frac{1}{2},$ so that,  \dref{ssddcz2.571hhjjjhhhhh51222ccvvhddfccvvhjjjksdddffkhhggjjllll} yields to
 \begin{equation}
\begin{array}{rl}
\|u_{\varepsilon}(\cdot, t)\|_{L^\infty(\Omega)}\leq  C_{29}~~ \mbox{for all}~~ t\in(0,T_{max,\varepsilon}).\\
\end{array}
\label{cz2.5jkkcvvvhjdsdfffffdkfffffkhhgll}
\end{equation}

{\bf Step 5. The boundedness of $\|\nabla u_{\varepsilon}(\cdot, t)\|_{L^p(\Omega)}$ (with $p>1$) for all $t\in (0, T_{max,\varepsilon})$}

To verify this, given $p > 1$ we choose  $\gamma\in(\frac{1}{2},1)$ suitably large fulfilling $\gamma > 1- \frac{1}{p}$. Then by
using \dref{zjscz2.5297x9630111kkhhffrreerr}, we derive from $D(A^\gamma) \hookrightarrow W^{1,p}(\Omega;\mathbb{R}^2)$ (see \cite{GHenryHenry4441215})  that \dref{zjscz2.5297x963011dkllldfff1kkhh} holds.

{\bf Step 6. The boundedness of $\|c_{\varepsilon}(\cdot, t)\|_{W^{1,\infty}(\Omega)}$ for all  $t\in (0, T_{max,\varepsilon})$}

In order to prove this, for any $\sigma\in (0, T_{max,\varepsilon})$ and $\sigma<1,$ it is thus sufficient to derive a bound
$$\|c_{\varepsilon}(\cdot, t)\|_{W^{1,\infty}(\Omega)}\leq  C_{30}~~ \mbox{for all}~~ t\in(\sigma,T_{max,\varepsilon})$$
and some positive constant $C_{30}$.
To achieve this,
we use well-known smoothing properties of the Neumann heat semigroup $(e^{t\Delta})_{t\geq0}$ in
$\Omega$, as stated e.g. in Lemma 1.3 of \cite{Winkler792} in a version covering the present situation, to see that there exists
$C_{31} > 0$ such that
\begin{equation}
\begin{array}{rl}
&\| c_\varepsilon(\cdot, t)\|_{W^{1,\infty}(\Omega)}\\
\leq&\disp{C_{30}\{1+\frac{1}{t-\tau}\}e^{\lambda(t-\tau)}\| c_\varepsilon(\cdot, \tau)\|_{L^{2}(\Omega)}+C_{30}\int_{\tau}^t\{1+({t-\tau})^{-\frac{1}{2}-\frac{1}{\bar{q}_0}}\}e^{\lambda(t-\tau)}\|n_\varepsilon(\cdot, \tau)\|_{L^{\bar{q}_0}(\Omega)}ds}\\
&\disp{+C_{30}\int_{\tau}^t\{1+({t-\tau})^{-\frac{1}{2}-\frac{1}{\bar{q}_0}}\}e^{\lambda(t-s)}\|u_\varepsilon(\cdot, s)\nabla c_\varepsilon(\cdot, s)\|_{L^{\bar{q}_0}(\Omega)}ds}\\
\leq&\disp{C_{31}+C_{30}\int_{\tau}^t\{1+({t-\tau})^{-\frac{1}{2}-\frac{1}{\bar{q}_0}}\}e^{\lambda(t-s)}\|u_\varepsilon(\cdot, s)\|_{L^{\infty}(\Omega)}\|\nabla c_\varepsilon(\cdot, s)\|_{L^{\bar{q}_0}(\Omega)}ds}\\
\leq&\disp{C_{32}~~ \mbox{for all}~~ t\in(\sigma,T_{max,\varepsilon}),}\\
\end{array}
\label{zjccffgssbhjcvvvbscz2.5297x96301ku}
\end{equation}
where $\bar{q}_0=\frac{2}{1-\alpha}$.
where  we have used  \dref{ccvvx1.731426677gg} as well as  \dref{cz2.5jkkcvddffgggghhdfffjjkvvhjkfffffkhhgll} and \dref{cz2.5jkkcvvvhjdsdfffffdkfffffkhhgll}.
 The proof of Lemma \ref{lemma45630hhuujj} is completed.
\end{proof}
By virtue of \dref{1.163072x} and Lemma \ref{lemma45630hhuujj}, the local-in-time solution can be extended to the global-intime
solution.
\begin{proposition}\label{lemma45630hhuujjtydfrjj} Let
$(n_\varepsilon, c_\varepsilon, u_\varepsilon, P_\varepsilon)_{\varepsilon\in(0,1)}$
 be classical solutions of \dref{1.1fghyuisda} constructed in Lemma
\ref{lemma70} on $[0, T_{max,\varepsilon})$.  Then the solution is global on $[0,\infty)$. Moreover, one can find 
$C > 0$
independent of $\varepsilon\in(0, 1)$ such that
\begin{equation}
\|n_\varepsilon(\cdot,t)\|_{L^\infty(\Omega)}  \leq C ~~\mbox{for all}~~ t\in(0,\infty)
\label{zjscz2.5297x9630111kkuu}
\end{equation}
as well as
\begin{equation}
\|c_\varepsilon(\cdot,t)\|_{W^{1,\infty}(\Omega)}  \leq C ~~\mbox{for all}~~ t\in(0,\infty)
\label{zjscz2.5297x9630111kkhhii}
\end{equation}
and
\begin{equation}
\|u_\varepsilon(\cdot,t)\|_{L^{\infty}(\Omega)}  \leq C ~~\mbox{for all}~~ t\in(0,\infty).
\label{zjscz2.5297x9630111kkhhffrroo}
\end{equation}
Moreover, for all  $\gamma\in ( \frac{1}{2}, 1)$,
we also have
\begin{equation}
\|A^\gamma u_\varepsilon(\cdot,t)\|_{L^{2}(\Omega)}  \leq C ~~\mbox{for all}~~ t\in(0,\infty).
\label{zjscz2.5297x9630111kkhhffrreerrpp}
\end{equation}
Furthermore,  for any $p > 1$, there exists $C(p) > 0$ satisfying
\begin{equation}
\|\nabla u_\varepsilon(\cdot,t)\|_{L^{p}(\Omega)}  \leq C(p)~~\mbox{for all}~~ t\in(0,\infty).
\label{zjscz2.5297x963011dkllddfgggldfff1kkhh}
\end{equation}
\end{proposition}

Again due to the regularity properties asserted by Lemma \ref{lemma45630hhuujj}, and due to the assumed
H\"{o}lder continuity of $n_0$, it follows from standard parabolic theory that $n_{\varepsilon}$ even satisfies estimates in
appropriate H\"{o}lder spaces:
\begin{lemma}\label{lemma45630hhuujjuuyy}
Let $\alpha>0$.
Then one can find $\mu\in(0, 1)$ such that for some $C > 0$
%
%
\begin{equation}
\|c_\varepsilon(\cdot,t)\|_{C^{\mu,\frac{\mu}{2}}(\Omega\times[t,t+1])}  \leq C ~~\mbox{for all}~~ t\in(0,\infty)
\label{zjscz2.5297x9630111kkhhiioo}
\end{equation}
as well as
\begin{equation}
\|u_\varepsilon(\cdot,t)\|_{C^{\mu,\frac{\mu}{2}}(\Omega\times[t,t+1])} \leq C ~~\mbox{for all}~~ t\in(0,\infty),
\label{zjscz2.5297x9630111kkhhffrroojj}
\end{equation}
and
\begin{equation}
\|n_\varepsilon(\cdot,t)\|_{C^{\mu,\frac{\mu}{2}}(\Omega\times[t,t+1])} \leq C ~~\mbox{for all}~~ t\in(0,\infty),
\label{zjscz2.5297x96dfgg30111kkhhffrroojj}
\end{equation}
Moreover, for any $\tau> 0$,
 there exists $C(\tau) > 0$ fulfilling
\begin{equation}
\|\nabla c_\varepsilon(\cdot,t)\|_{C^{\mu,\frac{\mu}{2}}(\Omega\times[t,t+1])} \leq C(\tau) ~~\mbox{for all}~~ t\in(\tau,\infty).
\label{1111zjscz2.5297x9630111kkhhffrreerrpphh}
\end{equation}
\end{lemma}
\begin{proof}
Firstly, let $g_\varepsilon(x, t) := -c_\varepsilon+n_{\varepsilon}-u_{\varepsilon}\cdot\nabla c_{\varepsilon}$. 
Then by Proposition \ref{lemma45630hhuujjtydfrjj}, we derive that
$ g_\varepsilon$ 
is bounded in $L^{\infty} (\Omega\times(0, T))$ for any  $\varepsilon\in(0,1)$,  we may invoke the standard parabolic regularity theory  to the second
equation  of \dref{1.1fghyuisda} and  infer that 
\dref{zjscz2.5297x9630111kkhhiioo} and \dref{1111zjscz2.5297x9630111kkhhffrreerrpphh} holds.
With the help of the Proposition \ref{lemma45630hhuujjtydfrjj} again, performing standard semigroup estimation techniques to the third  equation of \dref{1.1fghyuisda}, we can  get \dref{zjscz2.5297x9630111kkhhffrroojj}.
In order to  derive \dref{zjscz2.5297x96dfgg30111kkhhffrroojj}, we need to rewrite the first equation of \dref{1.1fghyuisda} as
$$n_{\varepsilon t}=\nabla\cdot a(x,t,n_{\varepsilon},\nabla n_{\varepsilon}) +b(x,t,\nabla n_{\varepsilon})$$
with boundary data $a(x,t,n_{\varepsilon},\nabla n_{\varepsilon}) \cdot\nu= 0$ on $\partial\Omega$, where $a(x,t,n_{\varepsilon},\nabla n_{\varepsilon})=\nabla n_{\varepsilon}-n_{\varepsilon}S_\varepsilon(x, n_{\varepsilon}, c_{\varepsilon})\nabla c_{\varepsilon}$  and $b(x,t,\nabla n_{\varepsilon})=-u_{\varepsilon}\cdot\nabla n_{\varepsilon}$, $(x, t) \in\Omega\times(0,\infty).$
Let $p := \nabla n_{\varepsilon}$. From the Young inequality and the boundedness of $n_{\varepsilon}$ and $u_{\varepsilon}$ already
obtained we can easily conclude that for any $T > 0$, there exist positive constants
$C_i = C_i(T)$ ($i = 1, 2, 3$) such that
$$a(x,t,n_{\varepsilon},p)\cdot p\geq\frac{1}{2}p^2-C_1|\nabla c_{\varepsilon}|^2$$
as well as
$$|a(x,t,n_{\varepsilon},p)|\leq|p|+C_2|\nabla c_{\varepsilon}|$$
and
$$|b(x,t,p)|\leq C_3+\frac{1}{2}p^2.$$
According to (3.5), $\nabla c_{\varepsilon}$ belongs to $L^\infty((0,\infty);L^2(\Omega))$. Consequently, by Theorem 1.3 of  \cite{Porzio555drr}, \dref{zjscz2.5297x96dfgg30111kkhhffrroojj} holds.
\end{proof}
%
%
%
%
%
%
%
%
%
%
%
%
%
%
%
%
%
%
%
%

\section{Estimates in $C^{2+¦È,1+\frac{\theta}{2}}$ for $u_\varepsilon$, $c_\varepsilon$  and $n_\varepsilon$}

In order to use the Aubin-Lions  Lemma (see Simon \cite{Simon}), we will need at least some regularity of the
time derivative of bounded solutions. The required estimates are very close to those in \cite{Wang23421215} (see also Winkler \cite{Winkler11215,Winkler51215}. We will state the results here and only give a
sketch of the proofs.

\begin{lemma}\label{lemma4dd5630hhuujjuuyy}
Let $\alpha>0$. There exists $\theta\in(0, 1)$ with the property that one can find $C > 0$ such that for any
$\varepsilon\in (0, 1)$ and $ t_0 > 0$,
%
\begin{equation}
\|u_\varepsilon(\cdot,t)\|_{C^{2+\theta,1+\frac{\theta}{2}}(\bar{\Omega}\times[t,t+1])} \leq C ~~\mbox{for all}~~ t\geq t_0,
\label{zjscz2.5297x9dddd630111kkhhffrroojj}
\end{equation}
\begin{equation}
\|c_\varepsilon(\cdot,t)\|_{C^{2+\theta,1+\frac{\theta}{2}}(\bar{\Omega}\times[t,t+1])} \leq C ~~\mbox{for all}~~ t\geq t_0
\label{zjscz2.5297x9dddffgddd630111kkhhffrroojj}
\end{equation}
as well as
\begin{equation}
\|u_\varepsilon(\cdot,t)\|_{L^{p}((t,t+1);W^{2,p}(\Omega)} +\|u_{\varepsilon t}(\cdot,t)\|_{L^{p}(\Omega\times(t,t+1))}\leq C ~~\mbox{for all}~~ t\geq t_0~~\mbox{and}~~p>1
\label{zjscz2.5297x9dddjkkkkddggkd630111kkhhffrroojj}
\end{equation}
and
\begin{equation}
\|c_\varepsilon(\cdot,t)\|_{L^{p}((t,t+1);W^{2,p}(\Omega)} +\|c_{\varepsilon t}(\cdot,t)\|_{L^{p}(\Omega\times(t,t+1))}\leq C ~~\mbox{for all}~~ t\geq t_0~~\mbox{and}~~p>1.
\label{zjscz2.5297x9dddjkkkkddggkd630111fghhkkhhffrroojj}
\end{equation}
%
\end{lemma}
\begin{proof}
Given $p > 2$, by using Proposition \ref{lemma45630hhuujjtydfrjj} along with the boundedness of $\nabla\phi$, we can find $C_1 > 0$, $C_2>0$ and $C_3>0$
such that for all $\varepsilon\in (0, 1)$ we have
\begin{equation}
\begin{array}{rl}
&\|(Y_{\varepsilon}u_{\varepsilon} \cdot \nabla)u_{\varepsilon}(\cdot,t)\|_{L^{p}(\Omega)}+ \|n_\varepsilon(\cdot,t)\nabla\phi\|_{L^{p}(\Omega)}\\
\leq& \|Y_{\varepsilon}u_{\varepsilon} \|_{L^{2p}(\Omega)}\|\nabla u_{\varepsilon}(\cdot,t)\|_{L^{2p}(\Omega)}+  \|n_\varepsilon(\cdot,t)\|_{L^{p}(\Omega)}\|\nabla\phi\|_{L^{\infty}(\Omega)}\\
\leq&  C_2\|\nabla u_{\varepsilon} (\cdot,t)\|_{L^{2}(\Omega)}\|\nabla u_{\varepsilon}(\cdot,t)\|_{L^{2p}(\Omega)}+ C_{1}\\
\leq& C_{3}~~~\mbox{for all}~~ t\geq t_0\\
\end{array}
\label{cz2.571hhhhh5122ddfddfffghhhhffgg2ccvvhddfccffgghhhhvvhjjjkkhhggjjllll}
\end{equation}
by using \dref{zjscz2.5297x963011dkllddfgggldfff1kkhh} and the Cauchy-Schwarz inequality.
Therefore, with the help of the maximal Sobolev regularity
estimates for the Stokes evolution equation, we derive that there exists a positive constanst $C_4$ which depends on $p$ and $ t_0$ such that
\begin{equation}
\|u_\varepsilon(\cdot,t)\|_{L^{p}((t,t+1);W^{2,p}(\Omega)} +\|u_{\varepsilon t}(\cdot,t)\|_{L^{p}(\Omega\times(t,t+1))}\leq C_4. ~~\mbox{for all}~~ t\geq t_0,
\label{zjscz2.5297x9dddjkkkkkd630111kkhhffrroojj}
\end{equation}
which in view of a known embedding result (\cite{AmannAmannmo1216}) implies that for all $ t_0 > 0$,
we can find $\theta_1\in (0, 1)$ and $C_5$ such  that
\begin{equation}
\|u_\varepsilon(\cdot,t)\|_{C^{1+\theta_1,\theta_1}(\Omega\times[t,t+2])} \leq C_5 ~~\mbox{for all}~~ t\geq\frac{ t_0}{2}.
\label{zjscz2.5297x9ddddfggdd630111kkhhffrfffroojj}
\end{equation}
Moreover,  from Proposition \ref{lemma45630hhuujjtydfrjj}, we obtain $\theta_2 \in(0, 1)$ and $C_6 > 0$ satisfying
\begin{equation}
\begin{array}{rl}
\|Y_{\varepsilon} u_{\varepsilon}(\cdot,t)\|_{C^{\theta_2,\theta_2}(\Omega\times[t,t+2])} \leq C_6 ~~\mbox{for all}~~ t\geq\frac{ t_0}{2}.\\
\end{array}
\label{cz2.571hhhhh5122dffgggdfddfffghhhhffgg2ccvvhddffddfccffgghhhhvvhjjjkkhhggjjllll}
\end{equation}
Here we have used the Sobolev imbedding theorem. Therefore, collecting \dref{zjscz2.5297x9ddddfggdd630111kkhhffrfffroojj}
and \dref{cz2.571hhhhh5122dffgggdfddfffghhhhffgg2ccvvhddffddfccffgghhhhvvhjjjkkhhggjjllll}, one can
find $C_7 > 0$ and $\theta_3$ satisfying
\begin{equation}
\begin{array}{rl}
\|(Y_{\varepsilon}u_{\varepsilon} \cdot \nabla)u_{\varepsilon}(\cdot,t)\|_{C^{\theta_3,\theta_3}(\Omega\times[t,t+2])}
\leq& C_{7}~~~\mbox{for all}~~ t\geq\frac{ t_0}{2},\\
\end{array}
\label{cz2.571hhhhhkkll5122ddfddfffghhhhffggdfgg2ccvvhddfccffgghhhhvvhjjjkkhhggjjllll}
\end{equation}
which together with \dref{zjscz2.5297x96dfgg30111kkhhffrroojj} yields \dref{zjscz2.5297x9dddd630111kkhhffrroojj} according to classical Schauder estimates for the Stokes evolution
problem (\cite{Solonnikov556}). Likewise, we can claim that \dref{zjscz2.5297x9dddffgddd630111kkhhffrroojj} and \dref{zjscz2.5297x9dddjkkkkddggkd630111fghhkkhhffrroojj}  hold.
\end{proof}
To prepare our subsequent compactness properties of
$(n_\varepsilon, c_\varepsilon, u_\varepsilon, P_\varepsilon)$ by means of the Aubin-Lions lemma (See  Simon \cite{Simon}), we conclude
the following regularity property with respect to the time variable by Proposition \ref{lemma45630hhuujjtydfrjj}. The idea comes from Lemma 5.1 of \cite{Zhengsdsd6} and  Lemmas 3.22--3.23 of \cite{Winkler11215}.
%

\begin{lemma}\label{lemmssddda45630hhuujjuuyytt}
Let $\alpha> 0$.
%
%
%
%
%
 Then
for all $t > 0$, there exists $C(t) > 0$ such that
\begin{equation}
\int_0^t\|n_{\varepsilon t}(\cdot,t)\|_{(W^{1,2}(\Omega))^*}^2dt  \leq C(t) ~~\mbox{for all}~~ t\in(0,\infty)~~\mbox{and}~~\varepsilon\in(0,1)
\label{zjscz2.5297x9630111kkhhiioott4}
\end{equation}
\end{lemma}
\begin{proof}
 Firstly, with the help of  Lemma \ref{lemma45630hhuujjuuyy}, for all $\varepsilon\in(0,1),$ we can fix a positive constants $C_1$ such that
 \begin{equation}
n_\varepsilon\leq C_1,|\nabla c_\varepsilon|  \leq C_1~~\mbox{and}~~|u_\varepsilon|  \leq C_1 ~~\mbox{in}~~ \Omega\times(0,\infty).
\label{33444gbhnzjscz2.5297x9630111kkhhiioo}
\end{equation}
Recalling \dref{x1.73142vghf48gg} and $n_{\varepsilon}\geq0$ in
 $\Omega\times(0,\infty)$,  
 we also derive that
 \begin{equation}
|S_\varepsilon(x,n_\varepsilon,c_\varepsilon)|\leq \frac{C_S}{(1+n_\varepsilon)^\alpha} \leq C_S~~\mbox{in}~~ \Omega\times(0,\infty)~~\mbox{for all}~~\varepsilon\in(0, 1).
\label{4566hnjgbhngbhnzjscz2.5297x9630111kkhhiioo}
\end{equation}
Taking ${ n_{\varepsilon}}$ as the test function for the first equation of
$\dref{1.1fghyuisda}$, combining with \dref{33444gbhnzjscz2.5297x9630111kkhhiioo}--\dref{4566hnjgbhngbhnzjscz2.5297x9630111kkhhiioo} and using $\nabla\cdot u_\varepsilon=0$, we derive 
 that

\begin{equation}
\begin{array}{rl}
\disp{\frac{1}{{2}}\frac{d}{dt}\|n_\varepsilon \|^{{2}}_{L^{{2}}(\Omega)}+\int_{\Omega} |\nabla n_\varepsilon|^2}=&\disp{-\int_\Omega  n_\varepsilon \nabla\cdot(n_\varepsilon S_\varepsilon(x, n_{\varepsilon}, c_{\varepsilon})
\nabla c_\varepsilon) }\\
=&\disp{ \int_\Omega    n_\varepsilon S_\varepsilon(x, n_{\varepsilon}, c_{\varepsilon})
\nabla n_\varepsilon\cdot\nabla c_\varepsilon}\\
\leq&\disp{ C_S\int_\Omega    n_\varepsilon
|\nabla n_\varepsilon||\nabla c_\varepsilon|}\\
\leq&\disp{\frac{1}{2}\int_{\Omega} |\nabla n_\varepsilon|^2+\frac{1}{2}C_S^2C_1^4 |\Omega|~~\mbox{for all}~~ t>0.}\\
\end{array}
\label{2222ttty3333cz2.5114114}
\end{equation}
 Rearranging and integrating over $(0,t)$ imply
\begin{equation}
\int_{0}^t\int_{\Omega} |\nabla n_{\varepsilon}|^2\leq C_2~~\mbox{for all}~~ t>0.
 \label{fgbhvbhnjmkvgcz2.5ghhjuyhddffhjjuiihjj}
\end{equation}
Now, testing the first equation by certain   $\varphi\in W^{1,2}(\Omega)$, we have
 \begin{equation}
\begin{array}{rl}
\disp\int_{\Omega}n_{\varepsilon t}(\cdot,t)\cdot\varphi =&\disp{\int_{\Omega}\left[\Delta n_{\varepsilon}-\nabla\cdot(n_{\varepsilon}S_\varepsilon(x, n_{\varepsilon}, c_{\varepsilon})\nabla c_{\varepsilon})-u_{\varepsilon}\cdot\nabla n_{\varepsilon}\right]\cdot\varphi}
\\
=&\disp{-\int_\Omega \nabla n_\varepsilon\cdot\nabla\varphi+\int_\Omega n_{\varepsilon} S_\varepsilon(x, n_{\varepsilon}, c_{\varepsilon})\nabla c_{\varepsilon}\cdot\nabla\varphi+\int_\Omega n_{\varepsilon}u_\varepsilon\cdot\nabla\varphi}\\
\leq&\disp{\|\nabla n_\varepsilon\|_{L^2(\Omega)}\|\nabla\varphi\|_{L^2(\Omega)}+C_1C_S\|\nabla c_{\varepsilon}\|_{L^2(\Omega)}\|\nabla\varphi\|_{L^2(\Omega)}}\\
&+\disp{C_1\|\nabla n_\varepsilon\|_{L^2(\Omega)}\|\varphi\|_{L^2(\Omega)}}\\
\leq&\disp{(\|\nabla n_\varepsilon\|_{L^2(\Omega)}+C_1C_S\|\nabla c_{\varepsilon}\|_{L^2(\Omega)}+C_1\|\nabla n_\varepsilon\|_{L^2(\Omega)})\|\varphi\|_{W^{1,2}(\Omega)}.}\\
\end{array}
\label{gbhncvbmdcfvgcz2.5ghju48}
\end{equation}
Therefore,  we conclude that
 \begin{equation}
\begin{array}{rl}
\disp\|n_{\varepsilon t}(\cdot,t)\|_{(W^{1,2}(\Omega))^*}^2 =&\disp{\sup_{\varphi\in W^{1,2}(\Omega),\|\varphi\|_{W^{1,2}(\Omega)}\leq1}\left|\int_{\Omega}n_{\varepsilon t}(\cdot,t)\varphi\right|^2}\\
\leq&\disp{2(1+C_1)^2\|\nabla n_\varepsilon\|_{L^2(\Omega)}^2+2C_1^2C_S^2\|\nabla c_{\varepsilon}\|_{L^2(\Omega)}^2.}\\
\end{array}
\label{ghhjjgbhncvbmdcfvgcz2.5ghju48}
\end{equation}
Recalling \dref{33444gbhnzjscz2.5297x9630111kkhhiioo} and \dref{fgbhvbhnjmkvgcz2.5ghhjuyhddffhjjuiihjj}, and integrating this inequality, we can finally get \dref{zjscz2.5297x9630111kkhhiioott4}.
\end{proof}

\section{ Passing to the limit}

In this section we consider convergence of solutions of approximate problem \dref{1.1fghyuisda} and
then prove Theorem \ref{theorem3}. In order to
achieve this, we will first ensure that it is a weak solution. And then by applying the standard
parabolic regularity and the classical Schauder estimates for the Stokes evolution, we will
show that it is sufficiently regular so as to be a classical solution.
\begin{lemma}\label{lemma45630223}
Assume that   $\alpha>0$. There exist $\theta\in (0,1), (\varepsilon_j)_{j\in \mathbb{N}}\subset (0, 1)$ and functions
\begin{equation}
 \left\{\begin{array}{ll}
 n\in C^{\theta,\frac{\theta}{2}}_{loc}(\bar{\Omega}\times[0,\infty))\cap C^{2+\theta,1+\frac{\theta}{2}}_{loc}(\bar{\Omega}\times(0,\infty)),\\
  c\in  C^{\theta,\frac{\theta}{2}}_{loc}(\bar{\Omega}\times[0,\infty))\cap C^{2+\theta,1+\frac{\theta}{2}}_{loc}(\bar{\Omega}\times(0,\infty)),\\
  u\in  C^{\theta,\frac{\theta}{2}}_{loc}(\bar{\Omega}\times[0,\infty);\mathbb{R}^2)\cap C^{2+\theta,1+\frac{\theta}{2}}_{loc}(\bar{\Omega}\times(0,\infty);\mathbb{R}^2),\\
  P\in  C^{1,0}(\bar{\Omega}\times(0,\infty))\\
   \end{array}\right.\label{1.ffhhh1hhhjjkdffggdfghyuisda}
\end{equation}
such that $n\geq 0$ and $c \geq0$ in $\Omega\times (0,\infty),$ that $\varepsilon_j\searrow 0$ as $j\rightarrow\infty$  and
\begin{equation}
 \left\{\begin{array}{ll}
 n_\varepsilon\rightarrow n~~\in C^{0}_{loc}(\bar{\Omega}\times[0,\infty)),\\
  c_\varepsilon\rightarrow c~~\in C^{0}_{loc}(\bar{\Omega}\times[0,\infty)),\\
 u_\varepsilon\rightarrow u~~\in C^{0}_{loc}(\bar{\Omega}\times[0,\infty);\mathbb{R}^2)\\
   \end{array}\right.
   and\label{1.ffgghhhhh1dffggdfghyuisda}
\end{equation}
as $\varepsilon=\varepsilon_j\searrow 0$, and that $(n,c,u,P)$ solves \dref{1.1} in the classical sense in $\Omega\times(0,\infty).$

\end{lemma}
\begin{proof}

In conjunction with \ref{lemma4dd5630hhuujjuuyy} and \ref{lemmssddda45630hhuujjuuyytt}, Proposition \ref{lemma45630hhuujjtydfrjj}  and the standard compactness arguments (see \cite{Simon}), we can thus
find a sequence $(\varepsilon_j)_{j\in \mathbb{N}}\subset (0, 1)$ such that $\varepsilon_j\searrow 0$ as $j\rightarrow\infty$, and such that
\begin{equation}
 \left\{\begin{array}{ll}
 n_\varepsilon\rightarrow n ~~\mbox{ in}~~ L^2_{loc}(\bar{\Omega}\times[0,\infty))~~\mbox{a.e.}~~ \mbox{in}~~ \Omega\times (0,\infty),\\
  \nabla n_\varepsilon\rightharpoonup \nabla n ~~\mbox{weakly in}~~ L^2_{loc}(\bar{\Omega}\times[0,\infty)),\\
  c_\varepsilon\rightarrow c ~~\mbox{ in}~~ L^2_{loc}(\bar{\Omega}\times[0,\infty))~~\mbox{a.e.}~~ \mbox{in}~~ \Omega\times (0,\infty),\\
  \nabla c_\varepsilon\rightarrow \nabla c  ~~\mbox{in}~~ C^0_{loc}(\bar{\Omega}\times[0,\infty)),\\
  \nabla c_\varepsilon\rightharpoonup \nabla c ~~\mbox{in}~~ L^{2}(\Omega\times(0,\infty)),\\
  u_\varepsilon\rightarrow u ~~~~\mbox{ in}~~ L^2_{loc}(\bar{\Omega}\times[0,\infty))~~\mbox{a.e.}~~ \mbox{in}~~ \Omega\times (0,\infty),\\
  D u_\varepsilon\rightharpoonup Du ~~\mbox{weakly star in}~~L^{\infty}(\Omega\times(0,\infty))\\
 u_\varepsilon\rightarrow u ~~\mbox{in}~~  C^{0}_{loc}(\bar{\Omega}\times[0,\infty))\\
   \end{array}\right.
   \label{1.ffgghhhhh1dffggdfghyuisddffffda}
\end{equation}
for some limit function $(n, c, u).$ Moreover, \dref{1.ffgghhhhh1dffggdfghyuisddffffda} implies that
\begin{equation}n_\varepsilon S_\varepsilon(x, n_{\varepsilon}, c_{\varepsilon})\nabla c_\varepsilon\rightarrow nS(x, n, c)\nabla c~~\mbox{a.e.}~~\mbox{in}~\Omega\times(0,\infty)
\label{1.1ddddfddfftffghhhtyygghhyujiiifgghhhgffgge6bhhjh66ccdf2345ddvbnmklllhyuisda}
\end{equation}
by using \dref{x1.73142vghf48rtgyhu}
and \dref{x1.73142vghf48gg}. Next we shall prove that $(n, c, u)$ is a weak solution of problem \dref{1.1}, in the natural sense as specified
in \cite{Zhengssdddd00} (see also \cite{Winkler31215}). To this end, we first note that clearly $n$ and $c$ inherit nonnegativity from $n_{\varepsilon}$ and $c_{\varepsilon}$, and that $\nabla\cdot u = 0$
a.e. in
 $\Omega\times (0,\infty)$ according to \dref{1.1fghyuisda} and \dref{1.ffgghhhhh1dffggdfghyuisddffffda}. Now,
testing the first equation by certain   $\varphi\in C^\infty_0(\bar{\Omega}\times[0,\infty))$, we have
\begin{equation}
\begin{array}{rl}\label{eqsssx45xx12112cddfffcgghh}
\disp{-\int_0^{\infty}\int_{\Omega}n_{\varepsilon}\varphi_t-\int_{\Omega}n_0\varphi(\cdot,0) }=&\disp{-
\int_0^{\infty}\int_{\Omega}\nabla n_{\varepsilon}\cdot\nabla\varphi+\int_0^{\infty}\int_{\Omega}n_{\varepsilon}
S_{\varepsilon}(x,n_{\varepsilon},c_{\varepsilon})\nabla c_{\varepsilon}\cdot\nabla\varphi}\\
&+\disp{\int_0^{\infty}\int_{\Omega}n_{\varepsilon}u_{\varepsilon}\cdot\nabla\varphi~~\mbox{for all}~~\varepsilon\in(0, 1).}\\
\end{array}
\end{equation}
Then \dref{1.ffgghhhhh1dffggdfghyuisddffffda}, \dref{1.1ddddfddfftffghhhtyygghhyujiiifgghhhgffgge6bhhjh66ccdf2345ddvbnmklllhyuisda}, \dref{eqsssx45xx12112cddfffcgghh} and the dominated convergence theorem enables
us to conclude
\begin{equation}
\begin{array}{rl}\label{eqx45xx12112ccgghh}
\disp{-\int_0^{{\infty}}\int_{\Omega}n\varphi_t-\int_{\Omega}n_0\varphi(\cdot,0) }=&\disp{-
\int_0^{\infty}\int_{\Omega}\nabla n\cdot\nabla\varphi+\int_0^{\infty}\int_{\Omega}n
S(x,n,c)\nabla c\cdot\nabla\varphi}\\
&+\disp{\int_0^{\infty}\int_{\Omega}nu\cdot\nabla\varphi}\\
\end{array}
\end{equation}
by a limit procedure. Along with a similar procedure applied to the second equation
in \dref{1.1fghyuisda}, we can get
\begin{equation}
\begin{array}{rl}\label{eqx45xx12112ccgffggghhjj}
\disp{-\int_0^{\infty}\int_{\Omega}c\varphi_t-\int_{\Omega}c_0\varphi(\cdot,0)  }=&\disp{-
\int_0^{\infty}\int_{\Omega}\nabla c\cdot\nabla\varphi-\int_0^{\infty}\int_{\Omega}c\varphi+\int_0^{\infty}\int_{\Omega}n\varphi+
\int_0^{\infty}\int_{\Omega}cu\cdot\nabla\varphi}\\
\end{array}
\end{equation}
Finally, given any $\bar{\varphi}\in C_0^{\infty} (\bar{\Omega}\times[0, \infty);\mathbb{R}^2)$ satisfying $\nabla\cdot\bar{\varphi}\equiv0$,
from \dref{1.1fghyuisda} we obtain
%
%
\begin{equation}
\begin{array}{rl}\label{eqx45xx12112ccgghhjjgghh}
\disp{-\int_0^{\infty}\int_{\Omega}u_{\varepsilon}\bar{\varphi}_t-\int_{\Omega}u_0\bar{\varphi}(\cdot,0) -\kappa
\int_0^{\infty}\int_{\Omega} Y_{\varepsilon}u_{\varepsilon}\otimes u_{\varepsilon}\cdot\nabla\bar{\varphi} }=&\disp{-
\int_0^{\infty}\int_{\Omega}\nabla u_{\varepsilon}\cdot\nabla\bar{\varphi}-
\int_0^{\infty}\int_{\Omega}n_{\varepsilon}\nabla\phi\cdot\bar{\varphi}}\\
\end{array}
\end{equation}
for all $\varepsilon\in(0, 1).$
Now, in view of \dref{1.ffgghhhhh1dffggdfghyuisddffffda},  thanks to the dominated convergence theorem this entails that
$$
Y_\varepsilon u_\varepsilon\rightarrow u ~~\mbox{in}~~ L_{loc}^2([0,\infty); L^2(\Omega))~~\mbox{as}~~\varepsilon = \varepsilon_j\searrow 0,
$$
which together  with \dref{1.ffgghhhhh1dffggdfghyuisddffffda} yields that
$$
\begin{array}{rl}
Y_{\varepsilon}u_{\varepsilon}\otimes u_{\varepsilon}\rightarrow u \otimes u ~\mbox{in}~L^1_{loc}(\bar{\Omega}\times[0,\infty))~\mbox{as}~\varepsilon=\varepsilon_j\searrow0,
\end{array}
$$
so that, taking $\varepsilon=\varepsilon_j\searrow 0$ and using \dref{eqx45xx12112ccgghhjjgghh},  \dref{ccvvx1.731426677gg} and \dref{1.ffgghhhhh1dffggdfghyuisddffffda} implies
\begin{equation}
\begin{array}{rl}\label{eqx45xx121eerrr12ccgghhjjgghh}
\disp{-\int_0^{\infty}\int_{\Omega}u\bar{\varphi}_t-\int_{\Omega}u_0\bar{\varphi}(\cdot,0) -\kappa
\int_0^{\infty}\int_{\Omega} u\otimes u\cdot\nabla\bar{\varphi} }=&\disp{-
\int_0^{\infty}\int_{\Omega}\nabla u\cdot\nabla\bar{\varphi}-
\int_0^{\infty}\int_{\Omega}n\nabla\phi\cdot\bar{\varphi}.}\\
\end{array}
\end{equation}
Collecting \dref{eqx45xx12112ccgghh}--\dref{eqx45xx12112ccgffggghhjj} and \dref{eqx45xx121eerrr12ccgghhjjgghh} and using \dref{1.1fghyuisda}, we derive that
$(n, c, u)$ can be complemented by some pressure function $P$ in such a way that $(n, c, u,p)$ is a weak solution of \dref{1.1},  in the natural weak sense consistent with  those  in \cite{Winkler31215} and
\cite{Winkler51215}. 
Thereupon, in light of  the standard parabolic  Schauder theory (\cite{Ladyzenskajaggk7101}) and the classical Schauder estimates for the Stokes evolution
problem (\cite{Solonnikov556}), we derive from \dref{1.1fghyuisda} that there exist
$\theta\in (0,1)$ and  $P\in  C^{1,0}(\bar{\Omega}\times(0,\infty))$ with the property that
\begin{equation}
 \left\{\begin{array}{ll}
 n\in C^{\theta,\frac{\theta}{2}}_{loc}(\bar{\Omega}\times[0,\infty))\cap C^{2+\theta,1+\frac{\theta}{2}}_{loc}(\bar{\Omega}\times(0,\infty)),\\
  c\in  C^{\theta,\frac{\theta}{2}}_{loc}(\bar{\Omega}\times[0,\infty))\cap C^{2+\theta,1+\frac{\theta}{2}}_{loc}(\bar{\Omega}\times(0,\infty)),\\
  u\in  C^{\theta,\frac{\theta}{2}}_{loc}(\bar{\Omega}\times[0,\infty);\mathbb{R}^2)\cap C^{2+\theta,1+\frac{\theta}{2}}_{loc}(\bar{\Omega}\times(0,\infty);\mathbb{R}^2).\\
   \end{array}\right.\label{1.ffhhhddd1hhhjjkdffggdfghyuisda}
\end{equation}
And hence $(n,c,u,P)$ solves \dref{1.1} in the classical sense in $\Omega\times(0,\infty).$ 

The proof of Lemma \ref{lemma45630223} is completed.
\end{proof}


In order to prove our main result, we now only have to collect the results prepared during this section:

{\bf Proof of Theorem  \ref{theorem3}.} The statement is evidently implied by Lemma \ref{lemma45630223}.


{\bf Acknowledgement}:
This work is partially supported by  the National Natural
Science Foundation of China (No. 11601215) and the Shandong Provincial
Science Foundation for Outstanding Youth (No. ZR2018JL005). 

\end{document}